\documentclass[11pt]{article}
\usepackage{bbm}
\usepackage{amsmath,amsthm,amssymb}
\usepackage{graphicx}
\usepackage{tikz}
\usepackage[utf8]{inputenc}
\usepackage{pgfplots}
\usepackage{hyperref}
\usepackage{setspace}
\usepackage{pict2e,color}
\usepackage{multirow}
\usepackage{booktabs}
\usepackage{amsfonts}
\usepackage{verbatim} %for adding comments
\usepackage{sectsty}
\usepackage{url}
\usepackage{subcaption}
\usepackage{empheq}
\usepackage[normalem]{ulem}
\usepackage[titletoc,title]{appendix}
\usepackage[flushleft]{threeparttable}
\usepackage{soul} % use this (many fancier options)
\usepackage{natbib}
\usepackage{algorithm,algorithmic}
\usepackage{etoolbox}
\preto\subequations{\ifhmode\unskip\fi}
\usepackage{booktabs}
\usepackage{multirow}
\theoremstyle{definition}

\newtheorem{theorem}{Theorem}
\newtheorem{lemma}{Lemma}
\newtheorem{proposition}{Proposition}

\allowdisplaybreaks

\title{Resource Distribution Under Spatiotemporal Uncertainty of Disease Spread: Stochastic versus Robust Approaches}
\author{Beste Basciftci\thanks{Corresponding author; Department of Business Analytics, Tippie College of Business, University of Iowa, Email: {\tt beste-basciftci@uiowa.edu};}~~~Xian Yu\thanks{Department of Industrial and Operations Engineering, University of Michigan at Ann Arbor, Email: {\tt yuxian@umich.edu};}~~~Siqian Shen\thanks{Department of Industrial and Operations Engineering, University of Michigan at Ann Arbor, Email: {\tt siqian@umich.edu}.}}
\date{}

\begin{document}
\graphicspath{{figures/}}

\maketitle

\begin{abstract}
%Speeding up testing and vaccination is essential for controlling the coronavirus disease 2019 (COVID-19) pandemic. 
We consider the problem of optimizing locations of distribution centers (DCs) and plans for distributing resources such as test kits and vaccines, under spatiotemporal uncertainties of disease spread and demand for the resources. We aim to balance the operational cost (including costs of deploying facilities, shipping, and storage) and quality of service (reflected by demand coverage), while ensuring equity and fairness of resource distribution across multiple populations. We compare a sample-based stochastic programming (SP) approach with a distributionally robust optimization (DRO) approach using a moment-based ambiguity set. Numerical studies are conducted on instances of distributing COVID-19 vaccines in the United States and test kits, to compare SP and DRO models with a deterministic formulation using estimated demand and with the current resource distribution plans implemented in the US. We demonstrate the results over distinct phases of the pandemic to estimate the cost and speed of resource distribution depending on scale and coverage, and show the ``demand-driven'' properties of the SP and DRO solutions. Our results further indicate that if the worst-case unmet demand is prioritized, then the DRO approach is preferred despite of its higher overall cost. Nevertheless, the SP approach can provide an intermediate plan under budgetary restrictions without significant compromises in demand coverage.
\end{abstract}

~

\section{Introduction}
\label{sec:intro}

With the rapid spread of the coronavirus disease 2019 (COVID-19), testing is central to planning response activities in all countries during and post the pandemic. Effective and efficient testing can lead to early outbreak detection, to quickly isolate and treat infected patients, guide people consciously performing social distancing, and also lock down certain areas/activities if needed. Establishing efficient testing systems involves distributing test kits to test centers while considering potential demand uncertainty for testing \citep[see, e.g.,][]{Lampariello2020TestDistribution, Santini2020TestAllocation}. A similar resource distribution problem arises in the production and distribution of COVID-19 vaccines that became available in December 2020 -- That is, upon needs and orders from different regions, a central government sends limited amounts of vaccines to local agents using the information of regional infection status. This problem was solved in the United States (US) during the first two quarters of Year 2021 for distributing COVID-19 vaccines and is relevant to many countries when vaccines become available globally, or for distributing resources to respond to future outbreaks of other diseases. 

In the US, policymakers seek efficient ways to distribute vaccines to states and jurisdictions, and then to local hospitals, clinics, pharmacies, schools, communities and so on. Given certain vaccine allocation policies, how to distribute vaccines from production sites or distribution centers (DCs) to downstream demand is challenging and emergent. Meanwhile, it is hard to quantify the cost and speed trade-off for distributing medical resources such as vaccines or test kits, as there exist hidden costs associated with unsatisfied demand that may come from high-risk population groups. Given the evolving pandemic and rapidly changing demand, it is also challenging to locate DCs and make shipping plans optimally to achieve the best trade-off. Moreover, the scale of the problem can be enormous, involving millions of vaccines (or test kits) and large-scale distributed demand for operating the system for the whole country. A scientific way to conduct vaccine and test-kit distribution needs to rely on mathematical models and algorithms for processing and utilizing information from large datasets about the COVID-19 infection trends and demand for testing/vaccination. 

In the supply chain literature, there has been a large body of research studies on how to site retail stores, optimize inventory and production, manage stock levels as demand for products fluctuates by season \citep{daskin2011network, shen2011reliable, snyder_2006, basciftci2020distributionally}. Among them, many focus on facility locations under uncertain demand, using stochastic or robust optimization approaches. Establishing a COVID-19 resource distribution system presents a set of similar challenges under various uncertainties, but in addition, it requires to incorporate  more decisions such as resource distribution, inventory control, and demand shortage control into the multi-period facility planning problem. 
In this work, we present a mathematical framework that encapsulates allocation and distribution stages of disease-control resources (e.g., vaccines and test kits) by considering spatiotemporal uncertainty in the demand and ambiguity of the probability distribution of demand over a multi-period horizon. We optimize the locations of DCs, their capacities, shipment amounts and inventory levels. As the provided framework is presented in a generic form, it aims to address various resource allocation problems during different phases of a pandemic by comparing  deterministic, stochastic and distributionally robust decision-making approaches. 

The main contributions of the paper are threefold. First, we combine mathematical programming and statistical learning, for optimally locating DCs for disease-control resources and deriving shipping plans, customized for counties and states with diverse demographics and disease spread patterns. Our research can be utilized at the national level to balance the distribution, or at the state or county level to facilitate local operations. Second, through data-informed location optimization, we can effectively identify the most critical and vulnerable groups to prioritize testing or receiving vaccines under constrained amounts of resources and as a result, can protect other population groups. Third, we conduct extensive numerical studies using real COVID-19 infection data in the US and in the State of Michigan, to compare different approaches in terms of the operational cost and speed of resource distribution, by testing a diverse set of instances having different scales and parameter settings. 

The rest of the paper is organized as follows. In Section \ref{sec:litReview}, we review the most relevant literature in disease control, facility location, and vaccine supply chain management. 
In Section \ref{sec:TwoStageStochasticModel}, we present the stochastic programming model based on a set of demand samples generated from a given probability distribution and in Section \ref{sec:dr-model}, we consider that the exact demand distribution is unknown and present the distributionally robust optimization model to optimize decisions against the distributional ambiguity. In Section \ref{sec:computation}, we conduct numerical studies using COVID-19 infection data to demonstrate results of distributing vaccines and test kits, and compare our solutions with what were implemented by the Centers for Disease Control and Prevention (CDC). In Section \ref{sec:concl}, we conclude the paper and present future research directions. 

\section{Literature Review}
\label{sec:litReview}

Resource distribution is of vital importance in many applications, in particular the ones related to disease control and disaster relief, as the resources are usually scarce \citep{Cao2012medicalResourceAlloc, Gupta2016DisasterReview}. 
%In post disaster relief, the time efficiency for distributing resources is also an important concern \citep{Gupta2016DisasterReview}. 
%, e.g., in ambulances dispatch \citep[see, e.g.,][]{Gong2007reliefOperations} and allocating medical resources to patients with deteriorating health conditions \citep[see, e.g.,][]{Xiang2016reliefOperations}. 
For COVID-19, personal protective equipment (PPE), test kits, hospital beds and ventilators, and, most recently, vaccines are among the resources that need to be effectively distributed at all levels \citep{Emanuel2020Allocation}.  \citet{bertsimas2020predictions} formulate a deterministic optimization model to improve ventilator allocation by allowing sharing of ventilators between hospitals from different states in the US. \citet{Billingham2020VentilatorAlloc} study a similar problem to optimize ventilator sharing. 
%Furthermore, hospital bed allocation has been considered by \citet{Collier2020ValueBasedResourceAlloc} through identifying ``hot spot" regions that have spikes in the number of infected cases and using a value-based optimization model that measures the value based on the usage of marginal medical resources. However, 
\citet{Lampariello2020TestDistribution} consider a single-period COVID-19 test-kit allocation problem by maximizing utility functions corresponding to disease detection capabilities in different regions, to determine the amount of test kits to be allocated in each region given a certain budget. 
\citet{Santini2020TestAllocation} considers the distribution of swabs and reagent to laboratories for maximizing the number of COVID-19 tests processed. The author formulates the problem as a deterministic integer programming model by considering sharing of swabs and reagent among different laboratories over a multi-period planing horizon. Although most of the existing literature consider different stages of resource planning and sharing during the pandemic, they use regression tools to forecast average demand, without explicitly modeling the spatiotemporal infection and demand uncertainties at the resource-planning phase.

Assuming demand being stochastic, \citet{Mehrotra2020pandemic,Blanco2020ResourceAlloc} propose stochastic programming approaches to address this issue by generating scenarios with different patterns to represent the randomness in the amount of patients due to uncertain disease spread, to re-allocate and share medical resources among different hospitals.
\citet{Yin2021_Ventilator} extends this problem to a risk-averse multi-stage stochastic programming setting by incorporating changing transmission dynamics. We note that these studies consider the resource allocation and redistribution between certain regions with existing facilities, and facility location decisions with their corresponding capacities have not been incorporated to the decision making process. 
Recently, \citet{Parker2020ResourceAlloc} use robust optimization by assuming an unknown number of patients for each specific day within a certain range from a nominal value, e.g., the average across a certain period. 
In addition to the studied settings, the probability distribution of the uncertain demand for resources may be ambiguous during emergency or disaster relief operations due to the inherent and abrupt nature of these events \citep[see, e.g.,][]{Liu2019DRO,Wang2020DRO}. This motivates the development of a distributionally robust optimization approach in our paper for robustly allocating resources during an unprecedented event such as the COVID-19 pandemic. To the best of our knowledge, this paper is the first work that leverages distributionally robust optimization for COVID-19 related medical resource distribution and facility location decision-making problems.

The supply chain operations of vaccines for infectious disease control involve vaccine production, allocation and distribution stages, with decisions on which vaccines to produce, how many doses to produce, who should be vaccinated, how the vaccines can be distributed, and so on  \citep[see][]{duijzer2018literature}. Equitable and timely allocation of vaccines to different population groups becomes necessary in order to eliminate infectious diseases worldwide \citep{Tebbens2009Vaccine}.  In that regard, \citet{Morton2017vaccine} provide a retrospective study over 2009 H1N1 pandemic by developing  deterministic optimization models to ensure fair and equitable allocation of vaccines through first determining coverage levels of each region and then finding an allocation plan within certain tolerance from the targeted  levels.  

\citet{golan2020vaccine} conduct a comprehensive literature review that focuses on the resilience of vaccine supply chains, and point out that the lack of network-based, modeling-based, quantitative analysis is a major gap that needs to be bridged in order to create methods of real-time analysis and decision tools for the COVID-19 vaccine supply chain. For COVID-19 specifically, \citet{bubar2020model} discuss strategies for vaccine prioritization among different population groups; \citet{babus2020optimal} estimate occupation-based infection risks and use age-based infection fatality rates in a model to assign priorities over populations with different occupations and ages; \citet{bertsimas2020optimizing} capture vaccine effects and the variability in mortality rates across sub-populations, and then integrate a predictive model into a prescriptive model to optimize vaccine allocation; \citet{simchiLevi2020vaccine} focus on how priority and population groups should be identified over time under limited supply. Nevertheless, as the majority of the literature is focusing on COVID-19 vaccine allocation, few studies develop mathematical approaches for solving the operational problems related to vaccine distribution. Among them, \citet{bertsimas2021locate} optimize vaccine site selection and the assignment of population to different sites while ensuring optimal subsequent vaccine allocation. They show that the proposed solution achieves critical fairness objectives and is also highly robust to uncertainties and forecasting errors. Different from most of the existing COVID-19 vaccine allocation studies, in this paper, we study a broader class of vaccine-related operational problems involves locating facilities to produce or store vaccines and shipping them to demand locations in a daily basis with inventory and unmet demand (i.e., lost sales or backlogging) being considered.

This study is closely related to the literature of facility location under uncertainty, which uses different approaches depending on the available information about the uncertain demand \citep{snyder_2006}. 
When the distribution of the uncertain parameter is known or can be estimated accurately, the facility location problem can be modeled as a stochastic program by optimizing the expected total cost in its risk-neutral setting \citep{shen2011reliable, BidhandiYusuff2011_StochasticFLP, Tolooie2020__StochasticFLP}. In contrast to the generic stochastic facility location models, in certain problem settings as in this study, operational decisions may need to be integrated into this strategic level problem \cite{Schutz2009, Georgiadis2011}. This study extends this line of research by determining facility location and capacity decisions with inventory, shipment and unmet demand amounts within a two-stage stochastic programming framework, while providing extensions of this model to address more complex settings that can involve inventory at distribution centers, lead time of shipping and different types of distribution centers. 
On the other hand, when the distribution of the uncertain parameter is not fully known (i.e. there exists ambiguity in the distribution), the decision maker can be more conservative by optimizing the problem over the worst-case distribution, which can be modeled using distributionally robust optimization \citep{Lu2015ReliableFLP, basciftci2020distributionally}. Although these studies consider uncertain demand when planning facilities, an integrated decision framework for optimizing the facility location decisions together with multi-period capacity allocation, inventory control and resource distribution plans has not been studied under distributional ambiguity of the multivariate demand, due to the modeling and computational complexity. Thus, this paper also presents contributions to the facility location literature by addressing this complex setting with a distributionally robust optimization approach. 

\section{A Stochastic Programming Approach}
\label{sec:TwoStageStochasticModel}

In this paper, we consider the distribution of vaccines or test kits in a given region as a capacitated facility location problem involving multiple periods of shipment planning. A decision maker may need to solve the problem for either national- or state-level operations, to distribute manufactured vaccines and test kits to different states or counties from located DCs. Here, the DCs could also be the manufacturing sites, as manufactures can dispatch productions to states or to counties directly. In addition, the federal or state governments can act to open up new DCs to improve the efficiency of operations and increase demand coverage. 

Denote $\mathcal{I},\ \mathcal{J},\ \mathcal{T}$ as the sets of potential sites for locating DCs, demand locations and finite periods, respectively. Let $c^o_i$, $c^h_{i}$, $c^s_{ijt}$, $c^u_{jt}$, $c^{I}_{jt}$ be the cost of operating  DC $i$, unit cost of installing capacities in DC $i$, unit shipping cost from DC $i$ to demand site $j$ in period $t$, unit penalty cost of unsatisfied demand and unit cost of inventory at demand site $j$ in period $t$, for all $i\in\mathcal{I},\ j\in\mathcal{J},\ t\in\mathcal{T}$, respectively. (Varying cost parameter $c^u_{jt}$ can help ensure fair resource distribution to prioritized demand locations based on their demographics and infection status over time.) Denote $B_t$ as the total capacity of manufacturing resources across all DCs in period $t$ for all $t\in\mathcal{T}$, determined by the total amount of raw materials, space, workers, etc., needed for manufacturing the resources during each period. Let $\boldsymbol{d}$ be the vector of uncertain demand (i.e., the number of people who need to be vaccinated or tested) and $P$ be its probability distribution. We use the Monte Carlo sampling approach \citep{kleywegt2002sample} to replace $P$ with an empirical distribution constructed using  $|\Omega|$ scenarios with each scenario $\omega\in\Omega$ having an equal probability $p^{\omega}=1/|\Omega|$. We consider a finite set $\Omega$ of realizations of the random vector $\boldsymbol{d}$. Specifically, for each scenario $\omega\in\Omega$, we use $d_{jt}(\omega)$ to represent the demand realization at site $j$ in period $t$ for all $j\in\mathcal{J}$ and $t\in\mathcal{T}$, and therefore $\boldsymbol{d} = [d_{jt}(\omega), \ \omega \in \Omega, \ j \in \mathcal J, \ t \in \mathcal T]^{\mathsf T}$. Throughout this paper, bold letters are used for representing vectors and matrices. For notational convenience, we also use notation $[m]$ to indicate set $\{1,\ldots,m\}$. 

We define binary variables $x_i\in\{0,1\}, \ \forall i\in \mathcal I$ such that $x_i=1$ if DC $i$ is built, and $x_i = 0$ otherwise. For each built DC $i$ and period $t\in \mathcal{T}$, we also decide its capacity $h_{it}\ge 0$ for manufacturing or storing  resources. Both variables $\boldsymbol{x}$ and $\boldsymbol{h}$ are planning decisions and their values need as determined before realizing uncertain demand $\boldsymbol{d}$. For each scenario $\omega\in\Omega$, we define variables $s_{ijt}(\omega)\ge 0$ as the amount of resources sent from DC $i$ to demand location $j$ in period $t$, for all $i \in \mathcal I, \ j \in \mathcal J, \ t \in \mathcal T$. For each demand site $j \in \mathcal J$, we allow to keep inventory if the received resources are more than the total demand or to back order otherwise. In particular, we assume that people who cannot receive vaccines or test kits in the current period will wait to be administered in future periods. Accordingly, we define variables $I_{jt}(\omega)\ge 0$ and $u_{jt}(\omega)\ge 0$ as the inventory and backlog recourse variables for each $j\in \mathcal{J}$, $t\in\mathcal{T}$, and $\omega \in \Omega$. Moreover, input parameters $I_{j0}$ and $u_{j0}$ denote the initial inventory and backlog at demand location $j$ for all $j\in\mathcal{J}$, whose values are the same across all scenarios. 

We first employ a two-stage stochastic mixed-integer linear programming  (SMIP) framework to formulate the problem, where in the first stage, we decide values of variables $x_i$ and $h_{it}$ for all $i \in \mathcal I, \ t\in \mathcal{T}$. In the second stage, given each demand value $d_{jt}(\omega)$, we optimize the corresponding shipping and inventory plans using variables $s_{ijt}(\omega)$, $I_{jt}(\omega)$ and $u_{jt}(\omega)$ for each scenario $\omega \in \Omega$. 

The SMIP model is given by: \small
\begin{subequations}\label{eq:TS-MILP}
\begin{align}
    \min\quad& \sum_{i\in\mathcal{I}}c^o_ix_i+\sum_{i\in\mathcal{I},t\in\mathcal{T}}c^h_i h_{it} +\sum_{\omega\in\Omega}p^{\omega}\left(\sum_{i\in\mathcal{I},j\in\mathcal{J},t\in\mathcal{T}}c^s_{ijt}s_{ijt}(\omega)+\sum_{j\in\mathcal{J},t\in\mathcal{T}}\left(c^I_{jt}I_{jt}(\omega)+c^u_{jt}u_{jt}(\omega)\right)\right) \label{eq:TS-MILP-Obj}\\
    \text{s.t.}\quad 
    & h_{it}\le M_ix_i,\ \forall i\in\mathcal{I},\ t\in\mathcal{T}, \label{eq:TS-MILP-ConstrSpatial}\\
    & \sum_{i\in\mathcal{I}} h_{it} \le B_t,\ \forall t\in\mathcal{T}, \label{eq:TS-MILP-ConstrTemporal}\\
    & \sum_{j\in\mathcal{J}}s_{ijt}(\omega) \leq h_{it},\ \forall i\in\mathcal{I},\ t\in\mathcal{T},\ \omega\in\Omega,  \label{eq:TS-MILP-ConstrShip}\\
    & \sum_{i\in\mathcal I}s_{ijt}(\omega)+I_{j (t-1)}(\omega) + u_{jt}(\omega) = d_{jt}(\omega) + I_{jt}(\omega) + u_{j (t-1)}(\omega), \ \forall j\in\mathcal{J},\ t\in\mathcal{T},\ \omega\in\Omega, \label{eq:TS-MILP-ConstrFlow}\\
    & x_{i}\in\{0,1\},\ h_{it},\ s_{ijt}(\omega),\ I_{jt}(\omega),\ u_{jt}(\omega)\ge 0,\ \forall i\in\mathcal{I},\ j\in\mathcal{J},\ t\in\mathcal{T},\ \omega\in\Omega,\label{eq:TS-MILP-noneg}
\end{align}
\end{subequations}
\normalsize
where the objective function \eqref{eq:TS-MILP-Obj} minimizes the total cost of opening DCs, installing their capacities, and the total expected cost of shipping, holding inventory and back-orders. Constraints \eqref{eq:TS-MILP-ConstrSpatial} prohibit assigning any capacity to a DC that is not in use, where $M_i$ is the capacity limit of the DC $i$ for all $i\in \mathcal I$. Constraints \eqref{eq:TS-MILP-ConstrTemporal} impose temporal limitation on the total production capacity of all DCs. Constraints \eqref{eq:TS-MILP-ConstrShip} link the first-stage variables with the second-stage recourse decisions such that the total shipment from each DC is no more than its installed capacity during any period in any scenario. Constraints \eqref{eq:TS-MILP-ConstrFlow} are ``flow-balance'' constraints to reflect the changes of inventory and back-order levels, depending on the amount of resources received and demand level at each site $j$, for each period $t$ in each scenario $\omega$. Constraints \eqref{eq:TS-MILP-noneg} require binary valued $x$-variables and set all the other variables to be non-negative. 

\subsection{Model extension I: Inventory at DCs}
When the costs of production capacity ($c^h_{it}$) and shipment ($c^s_{ijt}$) are time-independent, it is without loss of generality to assume that there is no inventory kept in DCs because those additional products can be shipped to and stored at customer sites. However, if these costs vary over time, it may be beneficial to keep some products in DCs and ship them in the future. In this case, we define additional variables $I^D_{it}(\omega)$ as the inventory at DC $i$ in period $t$ with scenario $\omega$ for all $i\in\mathcal{I},\ t\in\mathcal{T},\ \omega\in\Omega$. Accordingly, Constraints \eqref{eq:TS-MILP-ConstrShip} can be modified as follows:
\begin{align*}
    \sum_{j\in\mathcal{J}}s_{ijt}(\omega) + I^D_{it}(\omega) = h_{it} + I^D_{i(t-1)}(\omega),\ \forall i\in\mathcal{I},\ t\in\mathcal{T},\ \omega\in\Omega,
\end{align*}
where $I^D_{i0}(\omega) = I^D_{i0}, \ \forall \omega\in\Omega$ is the (given) initial inventory at DC $i$.
Moreover, the objective function \eqref{eq:TS-MILP-Obj} can be recast as
{\footnotesize
\begin{align*}
   \min \ \sum_{i\in\mathcal{I}}c^o_ix_i+\sum_{i\in\mathcal{I}, t\in\mathcal{T}}c^h_{it}h_{it} +\sum_{\omega\in\Omega}p^{\omega}\left(\sum_{i\in\mathcal{I},j\in\mathcal{J},t\in\mathcal{T}}c^s_{ijt}s_{ijt}(\omega)+\sum_{i\in\mathcal{I},t\in\mathcal{T}}c^{ID}_{it}I_{it}^D(\omega)+\sum_{j\in\mathcal{J},t\in\mathcal{T}}\left(c^I_{jt}I_{jt}(\omega)+c^u_{jt}u_{jt}(\omega)\right)\right), 
\end{align*}
}
\normalsize
where $c^{ID}_{it}$ is the unit inventory cost at DC $i$ in period $t$ for all $i\in\mathcal{I},\ t\in\mathcal{T}$.

\subsection{Model extension II: Lead time for shipping}
%\begin{remark}
%\label{remark:leadtime}
The SMIP model \eqref{eq:TS-MILP} assumes that there is no lead time when shipping from DCs to demand locations. In the case of a constant lead time for every DC and demand-location pairs, we can simply shift the optimal production and shipment plans accordingly. However, if the lead time varies by location, %the demand fluctuates dramatically over time,
it is worthwhile to add lead time into our model. We denote $L_{ij}$ as the lead time in sending resources from DC $i$ to customer location $j$ for all $i\in\mathcal{I},\ j\in\mathcal{J}$. Then, Constraints \eqref{eq:TS-MILP-ConstrFlow} can be adjusted to
\begin{align}
    \sum_{i\in\mathcal I: t \geq L_{ij}}s_{ij(t-L_{ij})}(\omega)+I_{jt-1}(\omega) + u_{jt}(\omega) = d_{jt}(\omega) + I_{jt}(\omega) + u_{jt-1}(\omega),\ \forall j\in\mathcal{J},\ t \in \mathcal{T}, \ \omega\in\Omega.
\end{align}
%\end{remark}

\normalsize
\subsection{Model extension III: Different types of DCs}
In practice, the total demand could be satisfied by multiple types of resources, and the DCs may have different processing, stocking and distributing costs, depending on their specialties. In such a setting, one can differentiate DCs in terms of their unit production, shipping and inventory costs and capacities by the types of resources they hold and distribute. For this purpose, we define $\mathcal{L}$ as the set of resource types with different requirements and let $x_{il}$ variable to be 1 if DC $i$ of type $l$ is built, and 0 otherwise. Similarly, we extend capacity $h_{itl}$, shipment $s_{ijtl}(\omega)$, inventory $I_{jtl}(\omega)$, unmet demand $u_{jtl}(\omega)$ decision variables and their associated costs to this setting by incorporating the type index $l$. 
We further introduce a new decision variable $\bar{d}_{jtl}(\omega)$ for representing the amount of the total demand in location $j$ at time $t$ satisfied by the product type $l$. 
The resulting extension of SMIP model \eqref{eq:TS-MILP} is given by: 
{\small
\begin{subequations}\label{eq:TS-MILP-DifferentDC}
\begin{align}
    \min\quad& \sum_{i\in\mathcal{I}, l\in\mathcal{L}} c^o_{il}x_{il} +\sum_{i\in\mathcal{I},t\in\mathcal{T}, l\in\mathcal{L}}c^h_{il} h_{itl} \notag \\
    & \qquad \qquad +\sum_{\omega\in\Omega}p^{\omega}\left(\sum_{i\in\mathcal{I},j\in\mathcal{J},t\in\mathcal{T}, l\in\mathcal{L}}c^s_{ijtl}s_{ijtl}(\omega) 
    +\sum_{j\in\mathcal{J},t\in\mathcal{T}, l\in\mathcal{L}} \left(c^I_{jtl}I_{jtl}(\omega)+  c^u_{jtl}u_{jtl}(\omega)\right)\right) \label{eq:TS-MILP-ObjDifferentDC}\\
    \text{s.t.}\quad 
    & h_{itl}\le M_{il}x_{il},\ \forall i\in\mathcal{I},\ t\in\mathcal{T}, \ l\in\mathcal{L},  \label{eq:TS-MILP-ConstrSpatialDifferentDC}\\
    & \sum_{i\in\mathcal{I}, l\in\mathcal{L}} h_{itl} \le B_t,\ \forall t\in\mathcal{T}, \label{eq:TS-MILP-ConstrTemporalDifferentDC}\\
    & \sum_{j\in\mathcal{J}}s_{ijtl}(\omega) \leq h_{itl},\ \forall i\in\mathcal{I},\ t\in\mathcal{T},\ l\in\mathcal{L}, \ \omega\in\Omega,  \label{eq:TS-MILP-ConstrShipDifferentDC}\\
    & \sum_{i\in\mathcal I}s_{ijtl}(\omega)+I_{j(t-1)l}(\omega) + u_{jtl}(\omega) = \bar{d}_{jtl}(\omega) + I_{jtl}(\omega) + u_{j(t-1)l}(\omega), \ \forall j\in\mathcal{J},\ t\in\mathcal{T},\ l\in\mathcal{L}, \ \omega\in\Omega, \label{eq:TS-MILP-ConstrFlowDifferentDC}\\
    & \sum_{l\in\mathcal{L}} \bar{d}_{jtl}(\omega) = d_{jt}(\omega),  \ \forall j\in\mathcal{J},\ t\in\mathcal{T},\ \omega\in\Omega,  \label{eq:TS-MILP-DemandDifferentTypesDifferentDC} \\
    & x_{il}\in\{0,1\},\ h_{itl},\ s_{ijtl}(\omega),\ I_{jtl}(\omega),\ u_{jtl}(\omega),\ \bar{d}_{jtl}(\omega)\ge 0, \ \forall i\in\mathcal{I},\ j\in\mathcal{J},\ l\in\mathcal{L}, \ t\in\mathcal{T},\ \omega\in\Omega.\label{eq:TS-MILP-nonegDifferentDC}
\end{align}
\end{subequations}}Constraints \eqref{eq:TS-MILP-ConstrFlowDifferentDC} satisfy the inventory and shipment relationship for each demand location, planning period, product type and demand realization, whereas Constraints  \eqref{eq:TS-MILP-DemandDifferentTypesDifferentDC} ensure  that the total demand is partitioned into different types of resources.

\section{A Distributionally Robust Optimization Approach}
\label{sec:dr-model}

The stochastic programming approach introduced in Section \ref{sec:TwoStageStochasticModel} assumes that the distribution of the underlying uncertainty is perfectly known and one can have access to a large amount of samples from the true distribution. However, these may not be true in practice, specifically during a pandemic, where the distribution may be misspecified under limited information. To model this type of distributional ambiguity, we consider a distributionally robust optimization model, in which optimal solutions are sought for the worst-case probability distribution within a family of candidate distributions, called an ``ambiguity set'' and  denoted by $\mathcal P$. We denote the random demand vector by $\boldsymbol{\xi} = [d_{jt}, \ j \in \mathcal{J}, \ t \in \mathcal{T}]^\top$ and the unknown probability distribution by $\mathbb P$. A distributionally robust analogous model of the SMIP model \eqref{eq:TS-MILP} is: 
\small
\begin{subequations} \label{eq:Type-1-DR-Prob}
\begin{align} 
    \min\quad & \sum_{i\in\mathcal{I}}c^o_ix_i+\sum_{i\in\mathcal{I},t\in\mathcal{T}}c^h_i h_{it}  + \max_{\mathbb{P} \in \mathcal{P}} \mathbb{E}[g(\boldsymbol{h},\boldsymbol{\xi})] \label{eq:Type-1-DR-Obj} \\
    \text{s.t.} \quad 
   & h_{it}\le M_ix_i,\ \forall i\in\mathcal{I},\ t\in\mathcal{T}, \label{eq:Type-1-DR-Capacity}\\
    & \sum_{i\in\mathcal{I}} h_{it} \le B_t,\ \forall t\in\mathcal{T}, \\
    & x_{i} \in \{0,1\}, \ h_{it}\ge 0,\ \forall i \in \mathcal{I}, \ t\in\mathcal{T},\label{eq:Type-1-DR-Nonneg}
    \end{align}
\end{subequations}
\normalsize
where \small
\begin{subequations}\label{eq:Type-1-DR-InnerProb}
\begin{align}
    g(\boldsymbol{h},\boldsymbol{\xi}) = \min\quad& \sum_{i\in\mathcal{I},j\in\mathcal{J},t\in\mathcal{T}}c^s_{ijt}s_{ijt} +\sum_{j\in\mathcal{J},t\in\mathcal{T}}\left(c^I_{jt}I_{jt}+c^u_{jt}u_{jt}\right) \label{eq:Type-1-DR-InnerProb-Obj}\\
    \text{s.t.}\quad 
    & \sum_{j\in\mathcal{J}}s_{ijt} \leq h_{it},\ \forall i\in\mathcal{I},\ t\in\mathcal{T}, \label{eq:TS-MILP-Type-1-DR-ConstrShip}\\
      & \sum_{i\in\mathcal I}s_{ijt}+I_{j(t-1)} + u_{jt} = \xi_{jt} + I_{jt} + u_{j(t-1)}, \ \forall j\in\mathcal{J},\ t\in\mathcal{T},\ \label{eq:TS-MILP-Type-1-DR-ConstrFlow}\\
    &s_{ijt},\ I_{j,t},\ u_{jt}\ge 0,\ \forall i\in\mathcal{I},\ j\in\mathcal{J},\ t\in\mathcal{T}.\ \label{eq:TS-MILP-Type-1-DR-noneg}
\end{align}
\end{subequations}
\normalsize
The objective function \eqref{eq:Type-1-DR-Obj} minimizes the cost of opening DCs, installing their capacities, and the worst-case expected cost under the set of candidate distributions in the ambiguity set $\mathcal P$. The inner problem \eqref{eq:Type-1-DR-InnerProb} corresponds to the second-stage problem and minimizes the total cost of shipping, holding inventory and backlogging given first-stage decision $\boldsymbol{h}$.

To express the form of uncertainty in distribution, two main classes of ambiguity sets can be defined for distributionally robust optimization models. They are (i) statistical distance-based ambiguity sets that consider distributions within a certain distance to a reference distribution \citep{Jiang2016,Esfahani2018}, and (ii) moment-based ambiguity sets that consider distributions based on moment information \citep{Delage2010,zhangambiguous}.
Since a reference distribution is essential in constructing the former class of ambiguity sets, such a distribution might be misleading in the case of a pandemic with abrupt and unprecedented changes in events and requires further data points to ensure its accuracy with a high level of confidence. Therefore, we focus on the moment-based ambiguity sets for describing possible distributions corresponding to the underlying demand uncertainty. 
To construct our ambiguity set, we bound a set of moment functions of $\boldsymbol{\xi}$ by certain parameters. Specifically, we consider $m$ different moment functions 
$\boldsymbol{f}:=(f_1(\boldsymbol{\xi}),\ldots,f_m(\boldsymbol{\xi}))^{\mathsf T}$. 
We assume that the random vector $\boldsymbol{\xi}$ has a finite support set $S$ containing possible realizations $S = \{\boldsymbol{\xi}^1, \ldots, \boldsymbol{\xi}^K\}$, for a given distribution $\mathbb{P}$. In this paper, we consider a finite support set for computational tractability. If using continuous support sets, the reformulation of the problem becomes as a semi-infinite integer program with an infinite number of constraints and  mixed-integer binary variables, and thus cannot be optimized directly using the state-of-the-art solvers. For example, \cite{zhang2018solving} develop approximation algorithms for solving 0-1 semi-definite program for the problem of distributionally robust surgery planning and we leave the continuous support case for our problem for future research but will test different sizes of the discrete support case to justify its proximity and result sensitivity.

Consequently, we reformulate $\mathbb{E}[g(\boldsymbol{h},\boldsymbol{\xi})]$ as  $\sum_{k\in[K]}p_kg(\boldsymbol{h},\boldsymbol{\xi}^k)$ and the inner problem~\eqref{eq:Type-1-DR-InnerProb} can be represented for each realization from the set $\{\boldsymbol{\xi}^1, \ldots, \boldsymbol{\xi}^K\}$.  
Then for each $s \in [m]$, the corresponding moment function $f_s(\boldsymbol{\xi}) = \prod_{j \in \mathcal{J},  t \in \mathcal{T}} \xi_{jt}^{k_{sjt}}$,
where $k_{sjt}$ is a non-negative integer indicating the power of $\xi_{jt}$ for the $s$-th moment function. 
The lower and upper bounds are defined by $\boldsymbol{l} := (l_1,\ldots, l_m)^{\mathsf T}$ and $ \boldsymbol{u} := (u_1,\ldots, u_m)^{\mathsf T}$, respectively. Correspondingly, we specify the ambiguity set $\mathcal P$ as follows. 
\begin{align}\label{eq:Type-1-Ambiguity}
\mathcal{P} :=\left\{\boldsymbol{p}\in\mathbb{R}_+^K\ |\ \boldsymbol{l}\le \sum_{k\in[K]}p_k\boldsymbol{f}(\boldsymbol{\xi}^k)\le\boldsymbol{u}\right\}.
\end{align}
Note that to guarantee that the ambiguity set~\eqref{eq:Type-1-Ambiguity} defines a set of probability distributions, one of the moment functions of $\boldsymbol{f}$ and its corresponding lower and upper bound values $l_s, u_s$, can be set as $\sum_{k\in [K]} p_{k} = 1$.
The following theorem demonstrates a reformulation of Model \eqref{eq:Type-1-DR-Prob} given the moment-based ambiguity set in \eqref{eq:Type-1-Ambiguity}.
\begin{theorem} 
\label{theorem:1}
If the ambiguity set defined in \eqref{eq:Type-1-Ambiguity} is non-empty, then Model \eqref{eq:Type-1-DR-Prob} can be reformulated as a single-level problem in the form: 
\begin{subequations}\label{Type-1-DR-Monolithic}
\begin{align}
\min\quad &\sum_{i\in\mathcal{I}}c^o_ix_i+\sum_{i\in\mathcal{I},t\in\mathcal T}c^h_{i}h_{it} -{\boldsymbol{\alpha}}^{\mathsf T} \boldsymbol{l}+{\boldsymbol{\beta}}^{\mathsf T} \boldsymbol{u} \\
\text{s.t.}\quad &\text{\eqref{eq:Type-1-DR-Capacity}--\eqref{eq:Type-1-DR-Nonneg}} \notag \\
&(-\boldsymbol{\alpha}+\boldsymbol{\beta})^{\mathsf T} \boldsymbol{f}(\boldsymbol{\xi}^k)\ge g(\boldsymbol{h},\boldsymbol{\xi}^k),\ \forall k\in[K], \label{eq:Type-1-DR-Monolithic-Constr} \\
& \boldsymbol{\alpha},\ \boldsymbol{\beta}\ge 0.
\end{align}
\end{subequations}
\end{theorem}

\begin{proof}
Explicitly stating the constraints in the ambiguity set $\mathcal P$, we first express the inner maximization problem as
\begin{subequations} \label{eq:Type-1-DR-Prob-InnerExpectation}
\begin{align} 
    \max \quad & \sum_{k \in [K]} p_k g(\boldsymbol{h},\boldsymbol{\xi}^k) \  \\
    \text{s.t.} \quad 
    & \sum_{k \in [K]} p_k \boldsymbol{f}(\boldsymbol{\xi}^k) \ge \boldsymbol{l}, \\
    &  \sum_{k \in [K]} p_k \boldsymbol{f}(\boldsymbol{\xi}^k) \le \boldsymbol{u}, \\
    & p_k \geq 0, \ \forall k \in [K]. 
    \end{align}
\end{subequations}
Letting $\boldsymbol{\alpha}$, $\boldsymbol{\beta}$ be the dual variables associated with the lower- and upper-bound constraints in \eqref{eq:Type-1-DR-Prob-InnerExpectation}, respectively, we obtain its dual formulation as 
\begin{subequations} \label{eq:Type-1-DR-Prob-InnerExpectationDual}
\begin{align}
\min_{\boldsymbol{\alpha},\boldsymbol{\beta}}\quad & -{\boldsymbol{\alpha}}^{\mathsf T} \boldsymbol{l}+{\boldsymbol{\beta}}^{\mathsf T} \boldsymbol{u} \\
\text{s.t.}\quad &(-\boldsymbol{\alpha}+\boldsymbol{\beta})^{\mathsf T} \boldsymbol{f}(\boldsymbol{\xi}^k)\ge g(\boldsymbol{h},\boldsymbol{\xi}^k),\ \forall k\in[K], \\
& \boldsymbol{\alpha},\ \boldsymbol{\beta}\ge 0.
\end{align}
\end{subequations}
Following the strong duality between  \eqref{eq:Type-1-DR-Prob-InnerExpectation} and \eqref{eq:Type-1-DR-Prob-InnerExpectationDual}, and replacing $\max_{\boldsymbol{p} \in \mathcal{P}} \sum_{k\in [K]}p_kg(\boldsymbol{h},\boldsymbol{\xi}^k)$ with \eqref{eq:Type-1-DR-Prob-InnerExpectationDual} in the outer optimization problem \eqref{eq:Type-1-DR-Prob}, we obtain the desired result. 
\end{proof}

Denote the empirical first and second moments of the uncertain demand parameter at location $j$ and period $t$ as $\mu_{jt}$ and $S_{jt}$, respectively. We examine a special form of ambiguity set \eqref{eq:Type-1-Ambiguity}, where for the first and second moments of each demand parameter, we consider their lower and upper bounds as follows:
\begin{subequations}
\label{eq:Type-1-Ambiguity-Special}
\begin{align}
\mathcal{P} =\Biggl \{\boldsymbol{p}\in\mathbb{R}_{+}^K\ |&\sum_{k\in [K]}p_{k}=1,\label{eq:Type-1-Ambiguity-Special-normal}\\
& \mu_{jt}-\epsilon_{jt}^{\mu}\le\sum_{k\in [K]} p_{k}\xi_{jt}^k\le \mu_{jt}+\epsilon_{jt}^{\mu},\ \forall j\in \mathcal{J},\ t\in\mathcal{T},\label{eq:Type-1-Ambiguity-Special-firstMoment}\\
& S_{jt}\underline{\epsilon}_{jt}^S\le\sum_{k\in [K]} p_{k}(\xi_{jt}^k)^2\le S_{jt}\bar{\epsilon}_{jt}^S,\ \forall j\in \mathcal{J},\ t\in\mathcal{T}\label{eq:Type-1-Ambiguity-Special-secondMoment}\Biggr \}.
\end{align}
\end{subequations}
Here, Constraint \eqref{eq:Type-1-Ambiguity-Special-normal} is a normalization constraint to ensure that $\{p_k\}_{k\in[K]}$ form a probability distribution. Constraints \eqref{eq:Type-1-Ambiguity-Special-firstMoment} bound the mean of parameter $\xi_{jt}$ in an $\epsilon_{jt}^{\mu}$-interval of the empirical mean $\mu_{jt}$, and constraints \eqref{eq:Type-1-Ambiguity-Special-secondMoment} bound the second moment of parameter $\xi_{jt}$ via scaling the empirical second moment $S_{jt}$ with parameters $0 \leq \underline{\epsilon}_{jt}^S \leq 1 \leq \bar{\epsilon}_{jt}^S$ for all $j\in\mathcal{J},\ t\in\mathcal{T}$. We note that under the perfect knowledge assumption of first and second moment information, the robustness parameters can be set as $\epsilon_{jt}^\mu = 0$, and $\underline{\epsilon}^\sigma_{jt}$ = $\overline{\epsilon}^\sigma_{jt}$ = 1, for all $j\in \mathcal{J},\ t\in\mathcal{T}$. Under the uncertainty of a pandemic, adjusting these parameters helps decision makers make inferences for representing spatiotemporal demand by considering different phases of the pandemic and taking into account potential deviations from predictive results. 

To obtain a mixed-integer linear programming reformulation of the single-level formulation~\eqref{Type-1-DR-Monolithic} given ambiguity set~\eqref{eq:Type-1-Ambiguity-Special}, we first describe some intermediate steps and results. In this regard, we analyze certain properties of the function $g(\boldsymbol{h},\boldsymbol{\xi}^k)$. 

\begin{proposition} \label{prop:InnerConvexPiecewise}
Function $g(\boldsymbol{h},\boldsymbol{\xi}^k)$ is a convex piecewise linear function in $\boldsymbol{h}$ for every $k \in [K]$. 
\end{proposition}

\begin{proof}
First, note that problem~\eqref{eq:Type-1-DR-InnerProb} is always feasible as it has complete recourse. Letting $\theta_{it}$, $\gamma_{jt}$ be the dual variables, we obtain the dual of problem~\eqref{eq:Type-1-DR-InnerProb} as follows. 
\begin{subequations} \label{eq:DualInnerProblem}
\begin{align} 
    \max\quad & \sum_{i\in\mathcal{I}} \sum_{t \in \mathcal{T}} h_{it} \theta_{it} + \sum_{j \in \mathcal{J}} \sum_{t = 2}^{T} \xi^k_{jt} \gamma_{jt} + \sum_{j \in \mathcal{J}} (\xi^k_{j1} + u_{j0} - I_{j0}) \gamma_{j1}  \ \label{eq:DualInnerProblem-Obj} \\
    \text{s.t.} \quad 
    & \theta_{it} + \gamma_{jt} \leq c^s_{ijt}, \ \forall  i \in \mathcal{I}, \ j \in \mathcal{J}, \ t \in \mathcal{T}, \label{eq:DualInnerProblem-ConstrFirst} \\
    & - \gamma_{jt} + \gamma_{j(t+1)} \leq c^{I}_{jt}, \ \forall j \in \mathcal{J}, \ t \in [T-1], \\
    & - \gamma_{jT} \leq c^{I}_{jT} \ \forall j \in \mathcal{J}, \\
     & \gamma_{jt} - \gamma_{j(t+1)} \leq c^{u}_{jt}, \ \forall j \in \mathcal{J}, \ t \in [T-1], \\
    & \gamma_{jT} \leq c^{u}_{jT} \ \forall j \in \mathcal{J}, \\
    & \theta_{it} \leq 0 \ \forall i \in \mathcal{I},\ t \in \mathcal{T}. \label{eq:DualInnerProblem-ConstrLast}
    \end{align}
\end{subequations}
As the dual problem \eqref{eq:DualInnerProblem} is feasible, at least one of the optimal solutions of this problem is at one of its extreme points. Considering the fact that objective function of this problem is linear in $\boldsymbol{h}$, the resulting optimal objective value can be represented as the maximum of the extreme-point-based function values. Hence, $g(\boldsymbol{h},\boldsymbol{\xi}^k)$ becomes a convex piecewise linear function in $\boldsymbol{h}$. 
\end{proof}

\begin{lemma} \label{lemmaEquivalentSets}
Denote $g(\boldsymbol{h},\boldsymbol{\xi}^k)$ by a compact form $\min_{\boldsymbol{y}} \{\boldsymbol{c}^\top \boldsymbol{y}: \boldsymbol{Ay} \geq \boldsymbol{h}, \boldsymbol{Dy} \geq \boldsymbol{f}\}$. Then, set $\Upsilon^1 := \{(\boldsymbol{h}, m): g(\boldsymbol{h},\boldsymbol{\xi}^k) \leq m\}$ for each $\boldsymbol{\xi}^k \in S$ has a polyhedral representation $\Upsilon^2 := \{(\boldsymbol{h}, m): \exists \ \boldsymbol{y} \ \text{s.t.} \ \boldsymbol{c}^\top \boldsymbol{y} \leq m, \ \boldsymbol{Ay} \geq \boldsymbol{h}, \ \boldsymbol{Dy} \geq \boldsymbol{f}\}$. 
\end{lemma}

\begin{proof} 
Demonstrated in Proposition \ref{prop:InnerConvexPiecewise}, $g(\boldsymbol{h},\boldsymbol{\xi}^k)$ is a convex function. Consider any $(\boldsymbol{h}, m) \in \Upsilon^1$. The optimal solution of $g(\boldsymbol{h},\boldsymbol{\xi}^k)$, say  $\boldsymbol{y}^1$, satisfies $\boldsymbol{c}^\top \boldsymbol{y}^1 \leq m, \boldsymbol{Ay}^1 \geq \boldsymbol{h}, \boldsymbol{Dy}^1 \geq \boldsymbol{f}$, proving that $(\boldsymbol{h}, m) \in \Upsilon^2$. For the other direction of the proof, consider any $(\boldsymbol{h}, m) \in \Upsilon^2$. Then, there exists a vector $\boldsymbol{y}^2$ such that it satisfies $\boldsymbol{c}^\top \boldsymbol{y}^2 \leq m, \boldsymbol{Ay}^2 \geq \boldsymbol{h}, \boldsymbol{Dy}^2 \geq \boldsymbol{f}$. As $\boldsymbol{y}^2$ is a feasible solution of the problem $\min_{\boldsymbol{y}} \{\boldsymbol{c}^\top \boldsymbol{y}: \boldsymbol{Ay} \geq \boldsymbol{h}, \boldsymbol{Dy} \geq \boldsymbol{f}\}$, $g(\boldsymbol{h},\boldsymbol{\xi}^k) \leq \boldsymbol{c}^\top \boldsymbol{y}^2 \leq m$, and therefore, $(\boldsymbol{h}, m) \in \Upsilon^1$. %This completes the proof. 
\end{proof}

Combining Lemma \ref{lemmaEquivalentSets} with the single-level formulation~\eqref{Type-1-DR-Monolithic}, we propose a mixed-integer linear programming reformulation of Model \eqref{eq:Type-1-DR-Prob} as follows.  

\begin{theorem}
Using the ambiguity set defined in \eqref{eq:Type-1-Ambiguity-Special}, Model \eqref{eq:Type-1-DR-Prob} is equivalent to the following mixed-integer linear program:
\begin{subequations}\label{eq:Type-1-DR-Monolithic-Special}
\begin{align}
\min\quad &\sum_{i\in\mathcal{I}}c^o_ix_i+\sum_{i\in\mathcal{I},t\in\mathcal T}c^h_{i} h_{it}-\alpha_{1}-\sum_{j\in\mathcal{J},t\in\mathcal{T}}\alpha_{2jt}(\mu_{jt}-\epsilon_{jt}^{\mu}) -\sum_{j\in\mathcal{J},t\in\mathcal{T}}\alpha_{3jt}(\mu_{jt}^2+\sigma_{jt}^2)\underline{\epsilon}^S_{jt}\nonumber\\
\quad&+\beta_{1}+\sum_{j\in\mathcal{J},t\in\mathcal{T}}\beta_{2jt}(\mu_{jt}+\epsilon_{jt}^{\mu})
+\sum_{j\in\mathcal{J},t\in\mathcal{T}}\beta_{3jt}(\mu_{jt}^2+\sigma_{jt}^2)\bar{\epsilon}^S_{jt}\\
\text{s.t.}\quad &\text{\eqref{eq:Type-1-DR-Capacity}--\eqref{eq:Type-1-DR-Nonneg}} \notag \\
&-\alpha_{1}+\beta_{1}+\sum_{j\in \mathcal J, t\in \mathcal{T}}\xi_{jt}^k(-\alpha_{2jt}+\beta_{2jt})+\sum_{j\in\mathcal{J},t\in\mathcal{T}}(\xi_{jt}^k)^2(-\alpha_{3jt}+\beta_{3jt})\ge \Phi^{k},\ \forall k \in [K], \\
    & \Phi^{k} = \sum_{i\in\mathcal{I},j\in\mathcal{J},t\in\mathcal{T}}c^s_{ijt}s^k_{ijt}  +\sum_{j\in\mathcal{J},t\in\mathcal{T}}\left(c^I_{jt}I^k_{jt}+c^u_{jt}u^k_{jt}\right),\ \forall k \in [K], \\
    & \sum_{j\in\mathcal{J}}s^k_{ijt} \leq h_{it},\ \forall i\in\mathcal{I},\ t\in\mathcal{T},\ \forall k \in [K], \\
      & \sum_{i\in\mathcal I} s^k_{ijt}+I^k_{j(t-1)} + u^k_{jt} = \xi^k_{jt} + I^k_{jt} + u^k_{j(t-1)}, \ \forall j\in\mathcal{J},\ t\in\mathcal{T},\ \forall k \in [K], \\
    & s^k_{ijt},\ I^k_{j,t},\ u^k_{jt}\ge 0,\ \forall i\in\mathcal{I},\ j\in\mathcal{J},\ t\in\mathcal{T}, \ \forall k \in [K], \\
    %&\text{\eqref{eq:TS-MILP-Type-1-DR-ConstrShip}--\eqref{eq:TS-MILP-Type-1-DR-noneg}} \notag \\
& \boldsymbol{\alpha},\ \boldsymbol{\beta}\ge \boldsymbol{0}.
\end{align}
\end{subequations}
\end{theorem}

\begin{proof}
To obtain the mixed-integer linear programming reformulation, we first revise the single-level formulation~\eqref{Type-1-DR-Monolithic} under the ambiguity set defined in \eqref{eq:Type-1-Ambiguity-Special}.
Then, we derive a polyhedral representation of the set $\{(\boldsymbol{h},m): g(\boldsymbol{h}, \boldsymbol{\xi}^k) \leq m\}$ where $m$ is the left-hand side of constraint \eqref{eq:Type-1-DR-Monolithic-Constr}. Using Lemma \ref{lemmaEquivalentSets} and the definition of the function $g(\boldsymbol{h}, \boldsymbol{\xi}^k)$, we derive the resulting formulation~\eqref{eq:Type-1-DR-Monolithic-Special}. 
\end{proof}

\section{Case Studies of COVID-19 Vaccine and Test Kit Distribution}
\label{sec:computation}

In this section, we present two comprehensive case studies of the presented optimization frameworks to conduct resource distribution under uncertain spatiotemporal demand for COVID-19 testing and vaccination. In Section \ref{sec:VaccineCaseStudy}, we consider COVID-19 vaccine distribution in the US, and in Section \ref{sec:TestkitCaseStudy}, we consider COVID-19 test kit distribution in the State of Michigan.  For both types of instances, we compute solutions of three different approaches, namely, the deterministic (DT) approach where the SMIP model \eqref{eq:TS-MILP} only contains one representative scenario in set $\Omega$, which can be the mean values of demand, the stochastic programming (SP) approach described in Section \ref{sec:TwoStageStochasticModel}, and the distributionally robust optimization (DRO) approach described in Section \ref{sec:dr-model}. For the SP and DRO approaches, we construct the optimization models based on 100 in-sample scenarios that are independently and identically generated from a given nominal distribution of the demand, and parameters in the nominal distribution follow estimated demand mean and variance based on real-world COVID-19 infection data in the corresponding regions of interest that we will describe later. Different solutions will be evaluated in out-of-sample scenarios to see their performance in terms of operational cost and possible lost sale/backlog (if demand values in the out-of-sample tests become extremely high). Note that the out-of-sample scenarios can be generated from the same nominal distribution for generating the in-sample data, or be a different one. The latter represents the case when future infection trends and thus future demand values become significantly different from what has been observed. We will test both cases in Section \ref{sec:sensitivity} to examine how the differences in data-generating distributions affect the solution performance. We will also describe details about how to generate out-of-sample scenarios and the results for the vaccine-distribution instances and test-kit distribution instances separately.

We use Gurobi 9.0.3 coded in Python 3.6.8 for solving all mixed-integer programming models. Our numerical tests are conducted on a Windows 2012 Server with 128 GB RAM and an Intel 2.2 GHz processor.

\subsection{Vaccine Distribution in the US}
\label{sec:VaccineCaseStudy}

\subsubsection{Experimental design and setup}
We test different vaccine-allocation phases following the government and CDC's guidelines, where in the earlier phases only healthcare workers and prioritized population groups (e.g., seniors more than 65 years old) are targeted for vaccination and in the later phases, larger segments of the population are recommended for vaccination to stop virus transmission. Specifically, we follow a recent epidemiological study by \citet{wang2020global} and consider three phases to distribute the vaccines, aiming to vaccinate 6.18\%, 41.97\%, 51.85\% of the overall adult population in the US, respectively. We also take into account certain levels of vaccine hesitancy in the population during each phase, and follow the results by \citet{malik2020determinants} who surveyed the US adult population to understand the acceptance of COVID-19 vaccines, among 10 Department of Health and Human Services (DHHS) regions listed below. %(We exclude remote islands from these regions and only focus on the mainland US, Alaska, Hawaii and Puerto Rico.) 
\begin{itemize}
    \item Region 1 -- Boston (Connecticut, Maine, Massachusetts, New Hampshire, Rhode Island, and Vermont); 
    \item Region 2 -- New York
(New Jersey, New York, and Puerto Rico);
\item Region
3 -- Philadelphia (Delaware, District of Columbia, Maryland, Pennsylvania, Virginia, and West Virginia);
\item Region 4 -- Atlanta (Alabama, Florida, Georgia, Kentucky, Mississippi, North Carolina, South Carolina,
and Tennessee);
\item Region 5 -- Chicago (Illinois, Indiana, Michigan, Minnesota, Ohio,
and Wisconsin);
\item Region 6 -- Dallas (Arkansas, Louisiana, New Mexico,
Oklahoma, and Texas); 
\item Region 7 -- Kansas City (Iowa, Kansas, Missouri,
and Nebraska);
\item Region 8 -- Denver (Colorado, Montana, North Dakota,
South Dakota, Utah, and Wyoming);
\item Region 9 -- San Francisco (Arizona,
California, Hawaii, Nevada);
\item Region 10 -- Seattle
(Alaska, Idaho, Oregon, and Washington).
\end{itemize}
We select the representative cities in the DHHS regions as our demand sites, which aggregate the demand in each region. Based on the acceptance rate estimates and sample sizes of the surveys conducted by \citet{malik2020determinants}, we compute a 90\% confidence interval of the acceptance rate for each region, depicted in Table \ref{tab:acceptance_rate}. We then multiply these values by the total adult population in each region to obtain lower and upper bounds on the number of people to be vaccinated. As each person needs to get two doses of vaccines, we multiply these bounds by 2 to determine the number of doses needed (i.e., potential demand). To represent the demand uncertainty, we sample scenarios following uniform distributions between the demand lower and upper bounds (i.e., $\underline{\mu}_{jt},\bar{\mu}_{jt}$). Table \ref{tab:demand_mean_vaccine} summarizes the mean values of the estimated demand during the three phases for each DHHS region. 

% Table generated by Excel2LaTeX from sheet 'vaccine_acceptance rate'
\begin{table}[ht!]
  \centering
  \caption{Confidence intervals of COVID-19 vaccines' acceptance rates in 10 DHHS regions}
  \resizebox{\textwidth}{!}{
    \begin{tabular}{lrrrrrrrrrr}
    \hline
    Regions & Region 1 & Region 2 & Region 3 & Region 4 & Region 5 & Region 6 & Region 7 & Region 8 & Region 9 & Region 10 \\
    \hline
    Acceptance & 68.06\% & 43.14\% & 72.04\% & 59.76\% & 39.13\% & 74.42\% & 72.22\% & 80.00\% & 67.61\% & 70.00\% \\
    90\% CI LB & 59.02\% & 31.73\% & 64.39\% & 50.85\% & 22.39\% & 68.95\% & 54.86\% & 65.29\% & 58.47\% & 60.99\% \\
    90\% CI UB & 77.09\% & 54.55\% & 79.70\% & 68.66\% & 55.87\% & 79.89\% & 89.59\% & 94.71\% & 76.74\% & 79.01\% \\
    \hline
    \end{tabular}%
    }
  \label{tab:acceptance_rate}%
\end{table}%

\begin{table}[ht!]
  \centering
  \caption{Estimated demand mean values during the three phases in 10 DHHS regions}
  \resizebox{\textwidth}{!}{
    \begin{tabular}{lrrrrrrrrrrr}
    \hline
    Regions & Region 1 & Region 2 & Region 3 & Region 4 & Region 5 & Region 6 & Region 7 & Region 8 & Region 9 & Region 10 & Total \\
    \hline
    Phase 1 & 1.0M  & 1.3M  & 2.2M  & 3.9M  & 2.0M  & 3.0M  & 970.0K & 920.4K & 3.3M  & 968.5K & 19.4M \\
    Phase 2 & 6.8M  & 9.0M  & 14.7M & 26.3M & 13.4M & 20.1M & 6.6M  & 6.3M  & 22.6M & 6.6M  & 132.4M \\
    Phase 3 & 8.4M  & 11.2M & 18.2M & 32.5M & 16.6M & 24.8M & 8.1M  & 7.7M  & 27.9M & 8.1M  & 163.5M \\
    \hline
    Total & 16M	& 22M &	35M &	63M &	32M	&48M&	16M	&15M&	54M&	16M	&316M\\
    \hline
    \end{tabular}%
    }
  \label{tab:demand_mean_vaccine}%
\end{table}%

As the sizes of populations to be vaccinated vary during different phases, we assume the lengths of the three phases to be 1, 2, and 3 months, respectively, with one period in our models being 2 weeks, and the demand of each phase is evenly distributed for each period. To determine candidate locations for siting DCs, we first consider the ones being used for vaccine production in the US currently. There are 5 DCs used by Pfizer-BioNTech and Moderna in the US. As these facilities are already available, we set their corresponding $\boldsymbol{x}$-values to 1 in all the models in the default case, and examine two more cases with this restriction being relaxed or strengthened later. We select representative cities of the 10 DHHS regions as potential locations to open additional DCs, and Table \ref{tab:locations} depicts all 15 DC locations and 10 demand locations. For each DC $i\in \mathcal{I}$, we calculate the capacity upper bound $M_i$ by assuming the daily maximum production to be 500,000, 750,000 and 1,000,000 doses for Phases 1, 2 and 3, respectively, as manufacturers will raise their daily production capacity as the demand increases. For each $t\in \mathcal{T}$, the temporal capacity of manufacturing vaccines $B_t$ is set as the sum of the maximum capacities over all DCs. We also examine the case when the supply chain of vaccines experiences certain disruptions and set the corresponding temporal capacities $B_t$ as 10\% of the normal capacities. The operating cost $c^o_i$ for each DC $i\in \mathcal{I}$ is estimated as 10000 times the local unit warehouse rental price per square foot, by assuming a standard warehouse is about 10000 square feet. The unit capacity cost $c^h_i$ is set as \$25 for each $i\in \mathcal I$, which is the average market price of one dose of COVID-19 vaccine. The unit inventory cost $c^I_{jt}$ is estimated as the sum of low-temperature inventory cost and energy cost of common refrigerators for storing vaccines, which is \$0.00008 for each $j\in \mathcal J,\ t\in \mathcal T$. The unit shipping cost $c^s_{ijt}$ for each region $j\in \mathcal J$ and supplier $i\in \mathcal I$ consists of two parts: The first part is the shipping cost for trucks calculated as \$3 per mile times the distance traveled from $i$ to $j$ in miles divided by 230,400, following the fact that each truck on average can carry 230,400 doses of Moderna vaccines \citep{link_moderna}, and the second part is the refrigerated trucks’ overall cost per liter of vaccine transported \citep{link_refrigated}.  Via sensitivity analysis of different penalty values for unit demand shortage, the penalty cost $c_{jt}^u$ of each unit of unsatisfied demand is set to \$100. The sources used for estimating these parameters are further summarized in Table \ref{tab:parameterSources} in the Appendix.

% Table generated by Excel2LaTeX from sheet 'suppliers'
\begin{table}[ht!]
  \centering
  \caption{Locations of 10 customers sites, 15 candidate DCs and optimal solutions for Phase 1 given by Deterministic (DT), Stochastic (SP) and Distributionally Robust Optimization (DRO) approaches}
    \begin{tabular}{cccccc}
    \hline
    City & Customer sites & Status of candidate DCs & DT & SP & DRO\\
    \hline
    Kalamazoo, MI &    & already opened &  &  & \\
    Pleasant Prairie, WI &    & already opened & & \checkmark & \\
    Bloomington, IN &    & already opened & & & \checkmark\\
    Norwood, MA &    & already opened & \checkmark & \checkmark & \checkmark\\
    Saint Louis, MO &    & already opened & & \checkmark & \\
    Boston, MA & Region 1   & candidate & & & \\
    New York City, NY & Region 2   & candidate & \checkmark & \checkmark & \\
    Philadelphia, PA & Region 3  & candidate & \checkmark & & \checkmark\\
    Atlanta, GA &  Region 4  & candidate & \checkmark & \checkmark & \checkmark\\
    Chicago, IL &  Region 5   & candidate & \checkmark & & \checkmark\\
    Dallas, TX & Region 6   & candidate & \checkmark & \checkmark & \checkmark\\
    Kansas City, KS & Region 7   & candidate & \checkmark & & \checkmark\\
    Denver, CO & Region 8   & candidate & \checkmark & & \checkmark\\
    San Francisco, CA & Region 9  & candidate & \checkmark & \checkmark & \\
    Seattle, WA &  Region 10  & candidate & \checkmark & & \checkmark\\
    \hline
    \end{tabular}%
  \label{tab:locations}%
\end{table}%

\subsubsection{Sensitivity Analysis}\label{sec:sensitivity}
We employ DT, SP, and DRO approaches to optimize facility location and resource distribution for Phases 1, 2, and 3, based on the same in-sample data. At the default setting, we randomly sample 100 scenarios (i.e., $K=100$) for the in-sample computation from a uniform distribution with the upper and lower bounds calculated based on Table \ref{tab:acceptance_rate}. For DT, we set demand mean values $\mu_{jt}, \ j\in \mathcal{J},\ t\in\mathcal{T}$ as the empirical mean values of the 100 in-sample scenarios. For DRO, we set parameter $\epsilon^{\mu}_{jt}$ in  \eqref{eq:Type-1-Ambiguity-Special-firstMoment} as $\epsilon^{\mu}_{jt}=0.5\mu_{jt}$, and bounding parameters in  \eqref{eq:Type-1-Ambiguity-Special-secondMoment} as   $\underline{\epsilon}^{S}_{jt}=0.1,\ \bar{\epsilon}^{S}_{jt}=2$ for all $j\in \mathcal{J},\ t\in\mathcal{T}$. The SP approach uses all the 100 scenarios to form set $\Omega$ in the SMIP model \eqref{eq:TS-MILP} and the DRO approach uses the in-sample scenarios as its support set for computing Phase 1's solutions. We independently generate 1000 out-of-sample scenarios for evaluating solution performance of the three approaches. In the last three columns in Table \ref{tab:locations}, we display the DCs that are open (in addition to the existing 5 DCs) in the optimal solutions of the three approaches for Phase 1's problem. 

To evaluate the impact of data-generating distributions, we present the in-sample and out-of-sample performance of the three approaches using normal and uniform distributions in Table \ref{tab:unshifted}, where in the normal distributions, we set the mean values to $(\underline{\mu}_{jt}+\bar{\mu}_{jt})/2$ and set the standard deviation to $(\bar{\mu}_{jt}-\underline{\mu}_{jt})/6$, respectively. The in-sample and out-of-sample scenarios are generated according to the same distributions in Table \ref{tab:unshifted}. We also record the percentage changes of the out-of-sample cost compared to the in-sample cost in the bracket. In practice, it is possible that the true data follows a distribution that is different than what we assume and therefore, we generate some out-of-sample scenarios by shifting an assumed distribution to the right. (That is, the true distributions have mean values of $1.01\mu_{jt}$ with $\mu_{jt}$ being the mean values of the assumed distributions.) The corresponding results are summarized in Table \ref{tab:shifted}. 

\begin{table}[ht!]
  \centering
  \caption{In-sample and out-of-sample performance comparison of the three approaches using different demand distribution assumptions and unshifted mean values.}
    \begin{tabular}{ccrrrr}
    \hline
    \multirow{2}[0]{*}{Distribution} & \multirow{2}[0]{*}{Approach} & \multicolumn{2}{c}{Unmet Demand} & \multicolumn{2}{c}{Overall Cost} \\
          &       & \multicolumn{1}{c}{In-sample} & \multicolumn{1}{c}{Out-of-sample} & \multicolumn{1}{c}{In-sample} & \multicolumn{1}{c}{Out-of-sample} \\
          \hline
    \multirow{3}[0]{*}{Normal} & DT    & 0     & 20K & 97M & 100M (+2.15\%)\\
          & SP    & 3K  & 3K  & 98M & 98M (+0.02\%)\\
          & DRO   & 0     & 34    & 99M & 99M (+0.003\%)\\
    \multirow{3}[0]{*}{Uniform} & DT    & 0     & 294K & 489M & 518M (+6.02\%)\\
          & SP    & 72K & 73K & 504M & 504M (+0.02\%)\\
          & DRO   & 6     & 752   & 520M & 520M (+0.01\%)\\
          \hline
    \end{tabular}%
  \label{tab:unshifted}%
\end{table}%

% Table generated by Excel2LaTeX from sheet 'In-sample100vaccine_biweekly'
\begin{table}[ht!]
  \centering
  \caption{In-sample and out-of-sample performance comparison of the three approaches using different demand distribution assumptions and shifted mean values.}
    \begin{tabular}{ccrrrr}
    \hline
    \multirow{2}[0]{*}{Distribution} & \multirow{2}[0]{*}{Approach} & \multicolumn{2}{c}{Unmet Demand} & \multicolumn{2}{c}{Overall Cost} \\
          &       & \multicolumn{1}{c}{In-sample} & \multicolumn{1}{c}{Out-of-sample} & \multicolumn{1}{c}{In-sample} & \multicolumn{1}{c}{Out-of-sample} \\
          \hline
    \multirow{3}[0]{*}{Normal} & DT    & 0     & 29K & 97M & 100M (+3.01\%) \\
          & SP    & 3K  & 5K  & 98M & 99M (+0.21\%) \\
          & DRO   & 1     & 28    & 99M & 99M (+0.003\%)\\
    \multirow{3}[0]{*}{Uniform} & DT    & 0     & 3M & 489M & 767M (+56.89\%)\\
          & SP    & 72K & 2M & 504M & 656M (+30.21\%) \\
          & DRO   & 6     & 784K & 520M & 599M (+15.08\%)\\
          \hline
    \end{tabular}%
  \label{tab:shifted}%
\end{table}%
From Table \ref{tab:unshifted}, the gaps between the out-of-sample and in-sample objective costs given by the DRO approach are the smallest across all three approaches and are always within 0.01\%, while the DT approach attains the highest gaps at 6.02\% when assuming uniformly distributed demand. These gaps further increase in Table \ref{tab:shifted} due to the inconsistency of assumed and true distributions. Notably, even with shifted means, the DRO approach can still achieve an objective cost gap of 0.003\% when assuming Normal distributed demand, showing the result robustness when having distributional ambiguity. Comparing the two demand distribution assumptions, following the Normal distribution generates much smaller in-sample and out-of-sample gaps. This is mainly because most of the scenarios generated by the Normal distribution are centered around the mean, and even with shifted means, the in-sample and out-of-sample scenarios are close to each other. However, when using the Uniform distribution to generate demand samples, scenarios are scattered between lower and upper bounds uniformly, and shifting the bounds can change scenarios significantly. Because of that, %we report vaccine allocation results based on uniform distributions in the rest of Section \ref{sec:VaccineCaseStudy}, and 
we only assume Normal distribution as the true demand distribution in the testing kit distribution problem in Section \ref{sec:TestkitCaseStudy}.

Next, we vary the number of in-sample scenarios $K$ from 10 to 100, and present the overall cost of the DRO approach on the same 1000 out-of-sample scenarios and the computational time comparison in Figure \ref{fig:VaryingK}. We also compare two different ambiguity sets in the DRO approach, where in ``First$+$Second Moment'', we use both first and second moments' information \eqref{eq:Type-1-Ambiguity-Special-firstMoment}--\eqref{eq:Type-1-Ambiguity-Special-secondMoment} to construct the ambiguity set, and in ``First Moment'', we only use the first moment's information \eqref{eq:Type-1-Ambiguity-Special-firstMoment}. 
\begin{figure}[ht!]
    \centering
    \begin{subfigure}[c]{.45\linewidth}
    \centering\includegraphics[width=\textwidth]{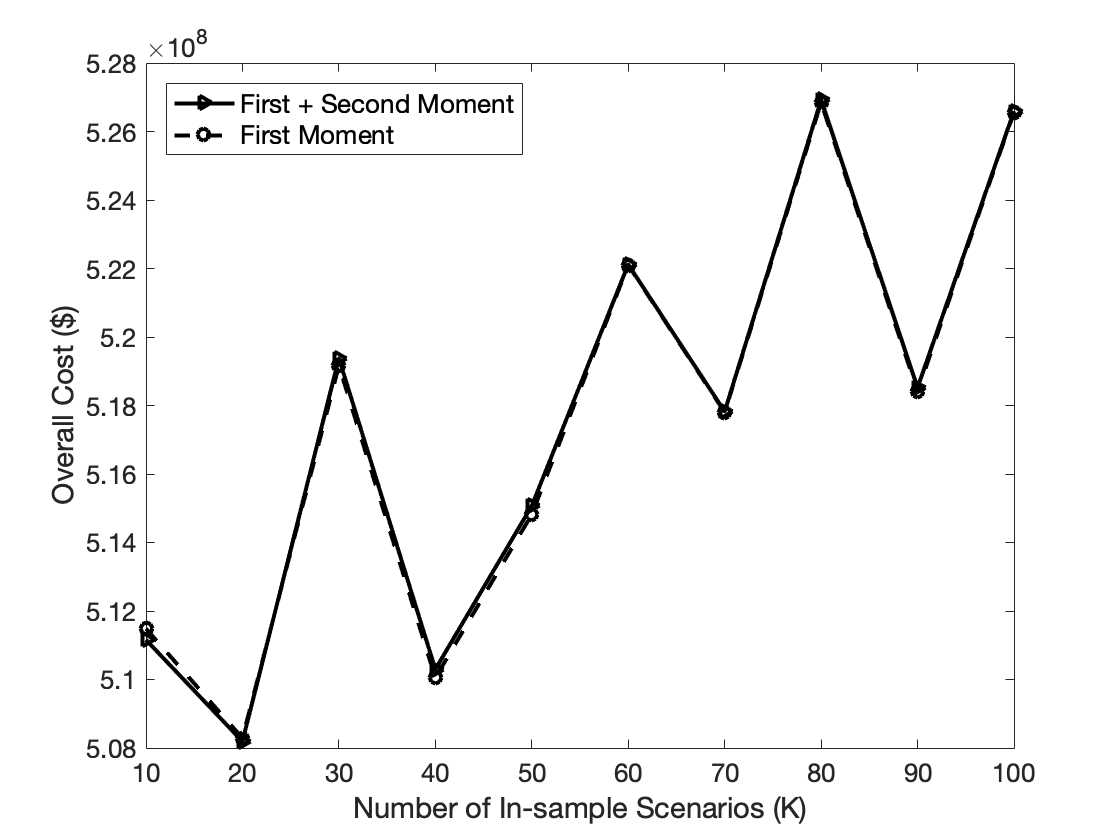}
    \caption{Out-of-sample overall cost comparison}\label{fig:VaryingK-cost}
    \end{subfigure}%
    \begin{subfigure}[c]{.45\linewidth}
    \centering\includegraphics[width=\textwidth]{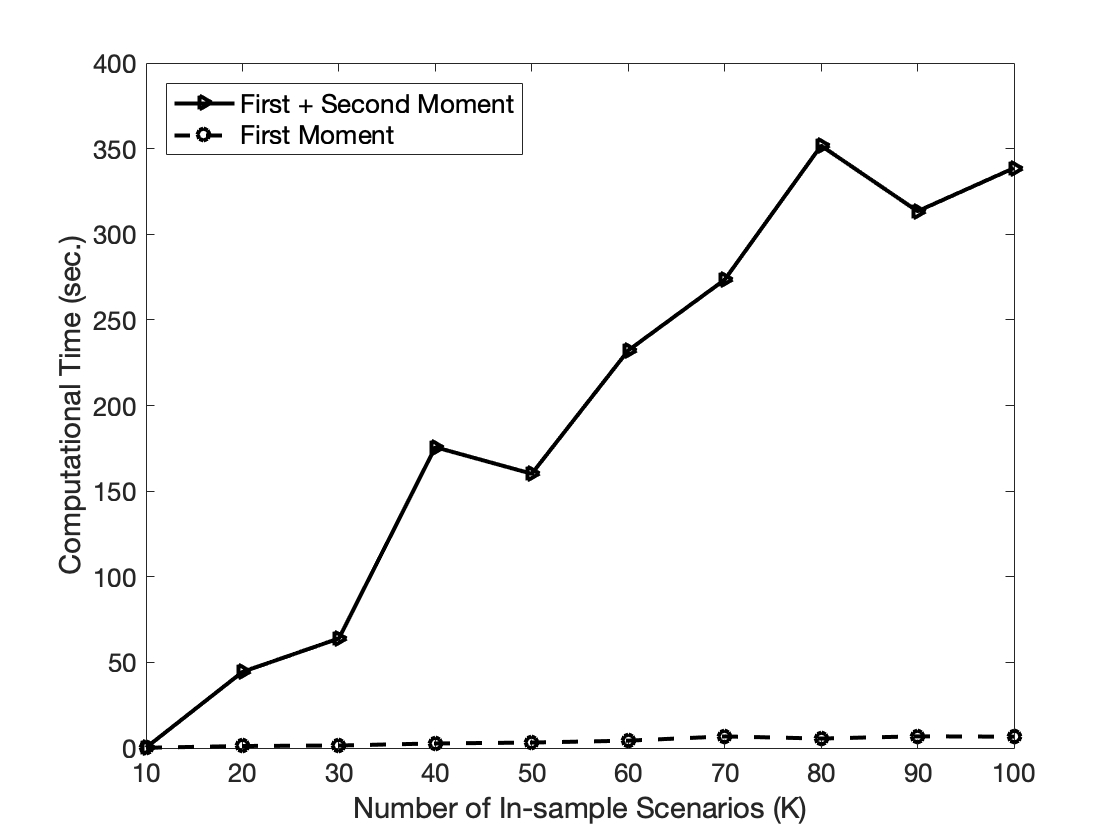}
    \caption{Computational time comparison}\label{fig:VaryingK-time}
    \end{subfigure}
\caption{Out-of-sample cost and CPU time of the DRO approach with varying $K$ over different ambiguity sets.}\label{fig:VaryingK}
\end{figure}
From Figure \ref{fig:VaryingK}, using only the first moment can obtain nearly the same out-of-sample performance as using both first and second moment information, and it scales better with the number of in-sample scenarios $K$. As we include more in-sample scenarios in the discrete support of the DRO approach, there are more choices for the inner-max problem to select the worst-case distribution and thus we become more conservative. Although the out-of-sample cost is fluctuating due to the differences between out-of-sample and in-sample scenarios, we can still see an increasing trend from Figure \ref{fig:VaryingK-cost}. In practice, to trade off between the computational complexity and conservatism, one can choose a reasonable $K$ between 10 and 100. In our test for solving Phases 2 and 3's problems, we set $K$ to 10 for DRO approach to produce solutions within a reasonable time limit.

We also examine the impact of the bounding parameters in the DRO approach, where we present the out-of-sample performance with varying $\epsilon_{\mu}$ in Figure \ref{fig:VaryingEpsilon-mu}, and the out-of-sample performance with varying $\underline{\epsilon}_S$ in Figure \ref{fig:VaryingEpsilon-S}, respectively.
\begin{figure}[ht!]
    \centering
    \begin{subfigure}[c]{.45\linewidth}
    \centering\includegraphics[width=\textwidth]{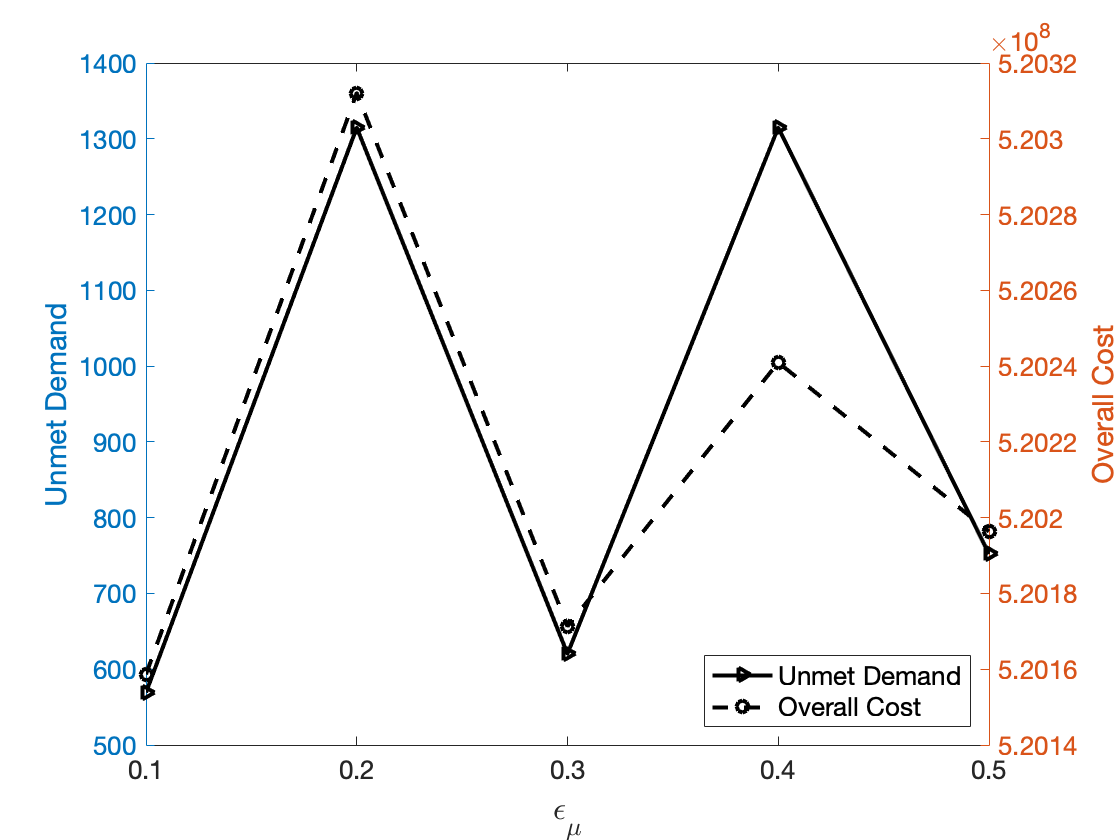}
    \caption{Results with varying $\epsilon_{\mu}$}\label{fig:VaryingEpsilon-mu}
    \end{subfigure}%
    \begin{subfigure}[c]{.45\linewidth}
    \centering\includegraphics[width=\textwidth]{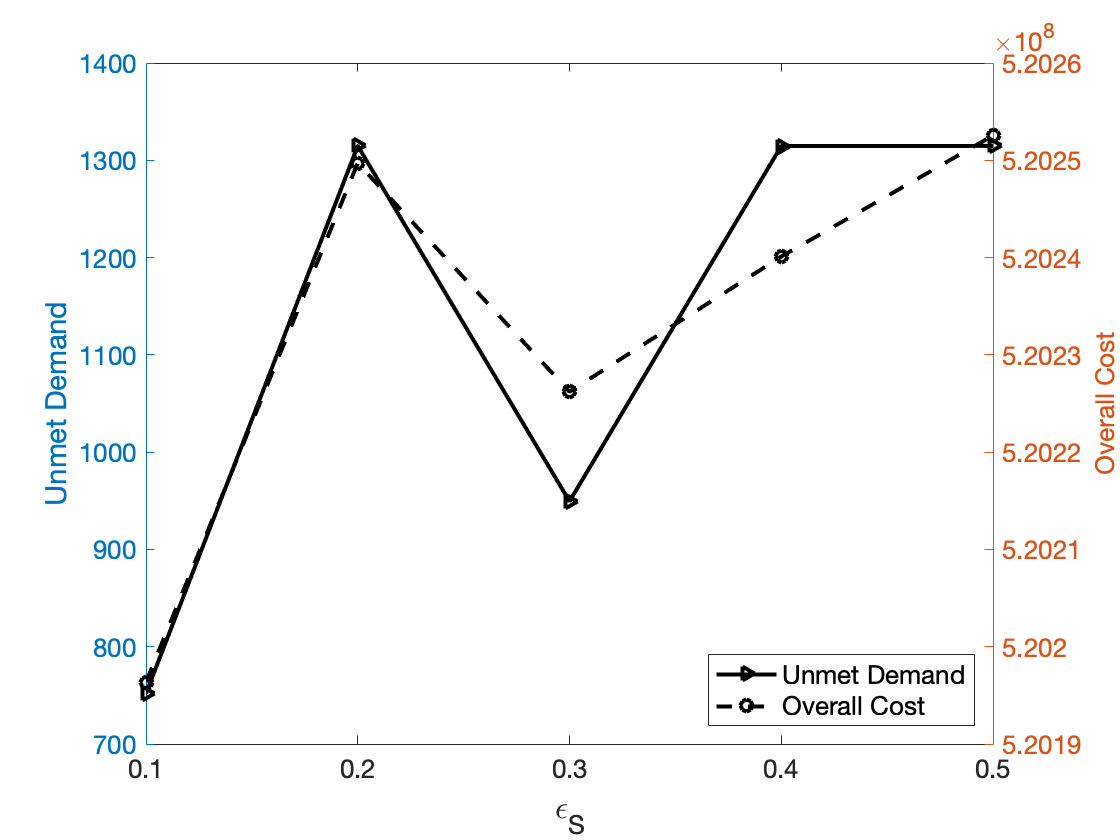}
    \caption{Results with varying $\epsilon_{S}$}\label{fig:VaryingEpsilon-S}
    \end{subfigure}
\caption{Out-of-sample cost and unmet demand changes with varying $\epsilon_{\mu}$ and $\epsilon_S$.}\label{fig:VaryingEpsilon}
\end{figure}
From Figure \ref{fig:VaryingEpsilon}, the out-of-sample cost is not very sensitive to both bounding parameters, as the percentage changes are all within $0.03\%$. When we increase $\underline{\epsilon}_S$, there is an increasing trend in both unmet demand and overall cost because we enlarge the ambiguity set and the results become more conservative.

\subsubsection{Results}

For solving problems in Phases 2 and 3, in order to obtain solutions for the DRO approach within a reasonable time limit (e.g., 24 hours), we reduce the number of in-sample scenarios to 10 to form the support set. Also, in our in-sample computation, we assume that the supply of vaccines matches the demand and therefore try to minimize the unmet demand only against its uncertainty rather than also taking into account resource scarcity. We present the unsatisfied demand and overall cost of different approaches under full and scarce resources (the latter being 10\% of the former) in Table \ref{tab:vaccine_summary}, where we mark the percentage increase compared with the optimal one in parentheses. The last column in Table \ref{tab:vaccine_summary} records the in-sample computational time in seconds for each method.

% Table generated by Excel2LaTeX from sheet 'Summary'
\begin{table}[ht!]
  \centering
     \caption{Out-of-sample performance of different approaches in terms of unmet demand and overall cost for Phases 1, 2, 3 of vaccine distribution in the US.}
     \resizebox{\textwidth}{!}{
    \begin{tabular}{cc|rr|rr|r}
    \hline
    \multirow{2}[0]{*}{Phases} & \multirow{2}[0]{*}{Approach} & \multicolumn{2}{c|}{Unmet Demand} & \multicolumn{2}{c|}{Overall Cost} & \multirow{2}[0]{*}{Time (sec.)}\\
          &       & \multicolumn{1}{r}{{Full Resource}} & \multicolumn{1}{r|}{{Scarce Resource}} & \multicolumn{1}{r}{{Full Resource}} & \multicolumn{1}{r|}{{Scarce Resource}} &  \\
          \hline
    \multirow{3}[0]{*}{Phase 1} & DT & {372K ($+1.26\times 10^4\%$)}     & {372K ($+5.31\times 10^3\%$)}     & {\$524M (+4.11\%)}     & {\$524M (+4.04\%)} & 0.1\\
          & SP    & {64K ($+2.11\times 10^{3}\%$)}     & {68K ($+8.85\times 10^2\%)$}     & {\$503M}     & {\$503M} & 6.5\\
          & DRO   & {3K}     & {7K}     & {\$516M (+2.50\%)}     & {\$513M (+1.95\%)} & 2148.8\\
          \hline
    \multirow{3}[0]{*}{Phase 2} & DT & {2M ($+2.37\times 10^3\%$)}    & {173M}     & {\$4B (+4.42\%)}     & {\$19B ($+8.13\times10^{-5}\%$)} & 0.1\\
          & SP    & {166K ($+77\%)$}     & {173M}     & {\$3B}     & {\$19B} & 1.3\\
          & DRO   & {94K}     & {173M}     & {\$3B (+0.40\%)}     & {\$19B ($+4.67\times10^{-5}\%$)} & 903.2\\
          \hline
    \multirow{3}[0]{*}{Phase 3} & DT & {3M ($+7.38\times 10^2\%$)}     & {131M}     & {\$4B (+4.40\%)}     & {\$16B ($+4.23\times10^{-4}\%$)} & 0.2\\
          & SP    & {313K}     & {131M}     & {\$4B}      & {\$16B} & 2.5 \\
          & DRO   & {450K (+43\%)}    & {131M}    & {\$4B (+0.39\%)}      & {\$16B ($+7.43\times10^{-4}\%$)} & 907.2\\
          \hline
    \end{tabular}%
    }
  \label{tab:vaccine_summary}%
\end{table}%

From Table \ref{tab:vaccine_summary}, in Phase 1, DRO obtains the least amount of unsatisfied demand and the second highest cost overall, while DT performs the worst in terms of both demand satisfaction and overall cost. In terms of the overall cost, SP always outperforms the other two approaches. This is because in the ambiguity set \eqref{eq:Type-1-Ambiguity-Special}, we allow deviations from the predictive results $\mu_{jt},\ S_{jt}$ and the worst-case scenario usually achieves a demand mean higher than the empirical mean $\mu_{jt}$. Compared to SP model, DRO always produces a higher manufacturing capacity against the worst-case scenario (as can be seen from Tables \ref{tab:vaccineCosts-ample}--\ref{tab:vaccineCosts-scarce}) and thus leads to a higher overall cost but lower unsatisfied demand. When facing scarce resource (i.e., 10\% of the regular capacity), the results do not change significantly in Phase 1 because all three approaches do not use up all the resources. However, when the resource capacity becomes tight (e.g., in Phases 2 and 3), the differences between the three approaches are almost negligible as there is not much flexibility in making different production and shipment plans due to limited temporal resources. We note that the raw cost breakdowns of each approach are detailed in Tables \ref{tab:vaccineCosts-ample} and \ref{tab:vaccineCosts-scarce} in the Appendix. % \ref{sec:Appendix-costs}.
Next, we focus on the analysis of optimal solutions of different methods under the setting of full resources.  

\begin{figure}[ht!]
    \centering
    \includegraphics[width=\textwidth]{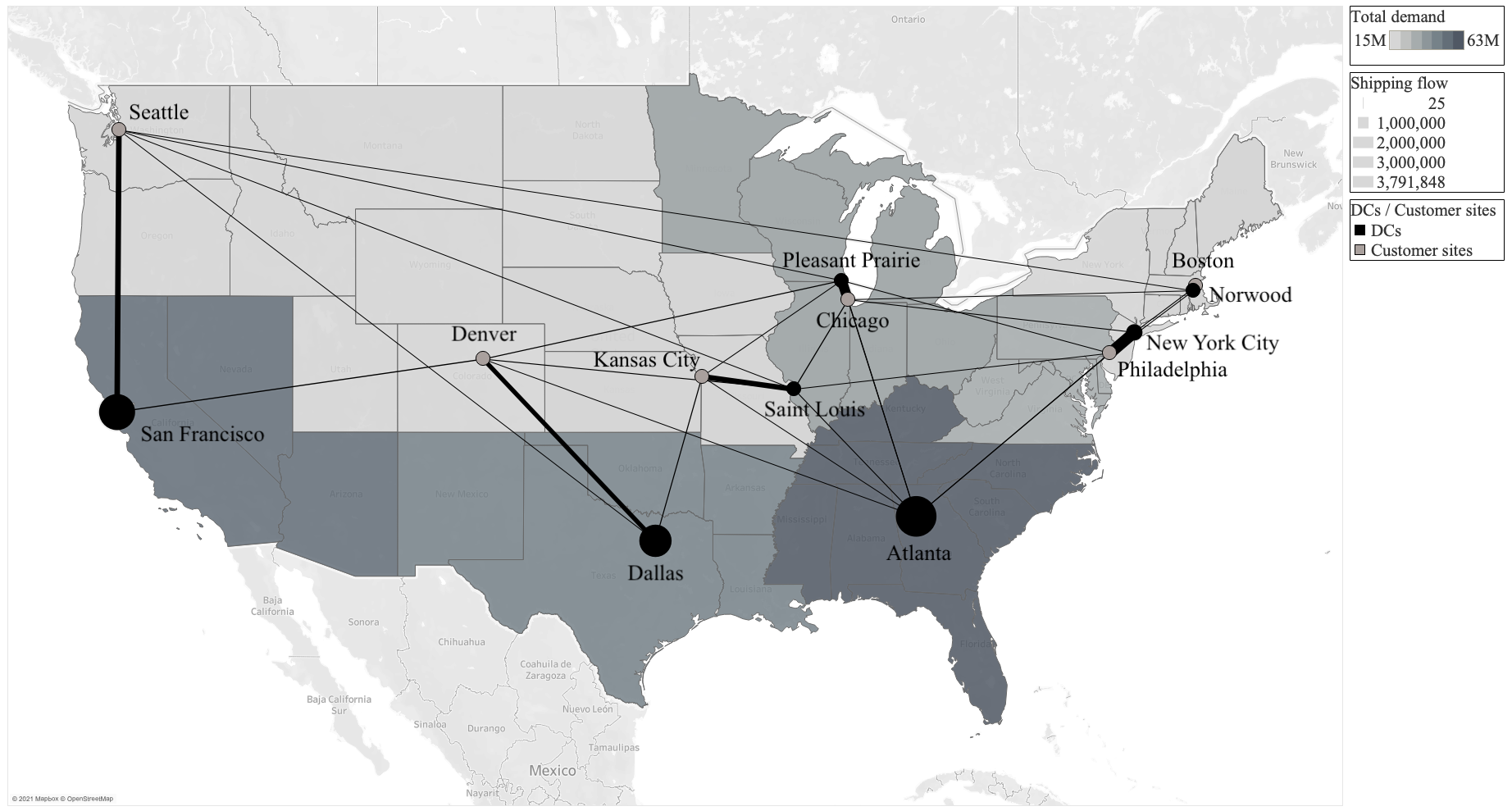}
    \caption{Average shipments of Phase 1's vaccines in 1000 out-of-sample scenarios given by the optimal solution of the SP approach. The background color of 10 DHHS regions indicates the total demand across all three phases where darker color represents higher demand volume. The black and gray circles stand for open DCs and demand sites, respectively, while DCs located in New York City, Atlanta, Dallas and San Francisco are also demand sites themselves and the sizes of these black circles represent the amount of vaccines they produce for satisfying their own demand. The black lines represent shipments from open DCs to demand locations with wider lines meaning higher volumes.}
    \label{fig:vaccine_flow_phase1_sp}
\end{figure}

\begin{figure}[ht!]
    \centering
    \includegraphics[width=\textwidth]{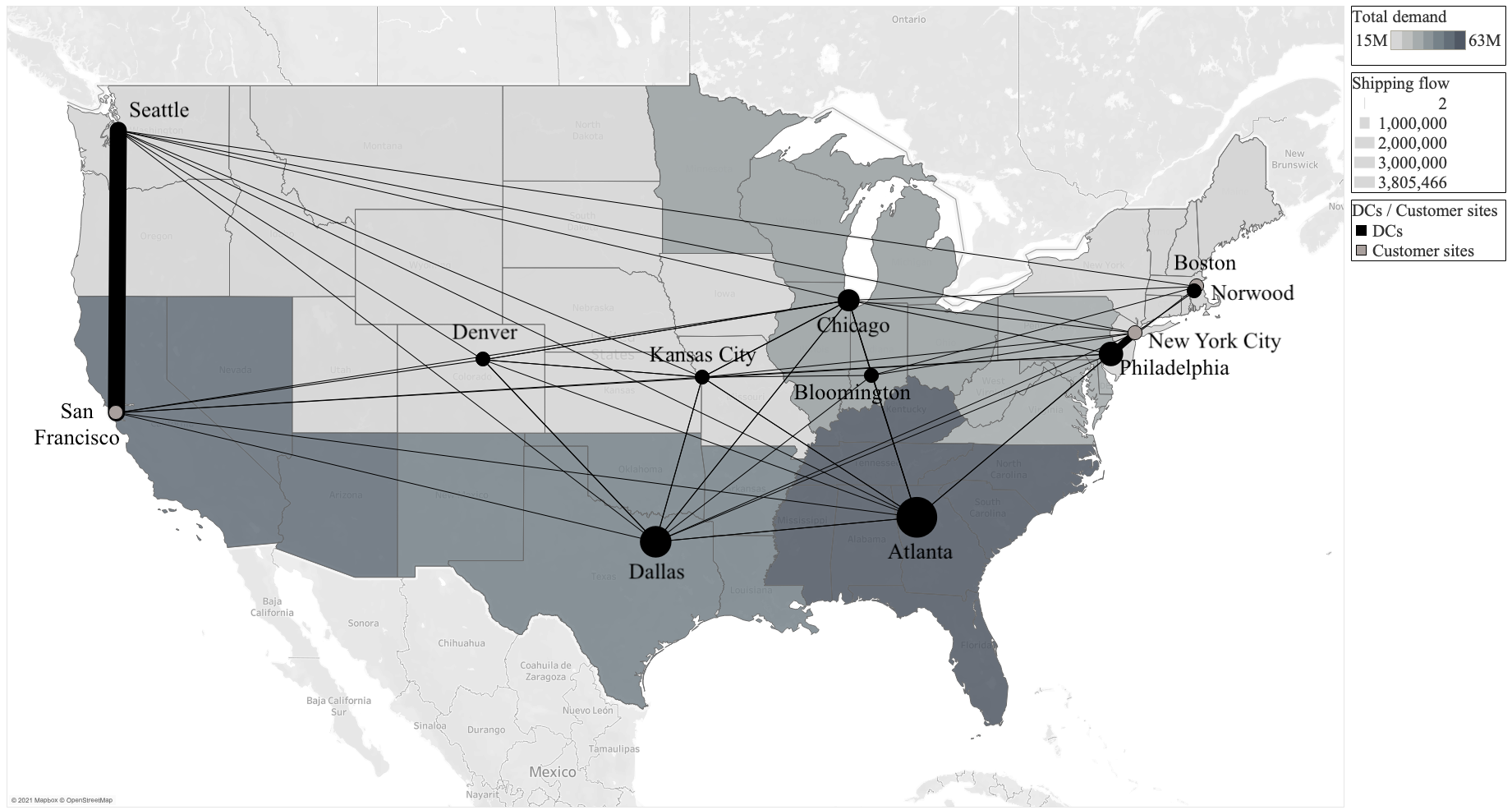}
    \caption{Average shipments of Phase 1's vaccines in 1000 out-of-sample scenarios given by the optimal solutions of the DRO approach. The background color of 10 DHHS regions indicates the total demand across all three phases where darker color represents higher demand volume. The black and gray circles stand for open DCs and customer sites, respectively, while DCs located in Philadelphia, Chicago, Atlanta, Dallas, Kansas City, Denver and Seattle are also demand sites themselves and the size of these black circles represent the amount of vaccines they produced for satisfying their own demand. The black lines represent shipments from open DCs to demand locations with wider lines meaning higher volumes. }
    \label{fig:vaccine_flow_phase1_dro}
\end{figure}

The average shipments in the optimal solutions of SP and DRO for Phase 1's vaccine distribution are presented in Figures \ref{fig:vaccine_flow_phase1_sp} and \ref{fig:vaccine_flow_phase1_dro}, respectively. From Figure \ref{fig:vaccine_flow_phase1_sp}, seven DCs are selected in the optimal solution, and among them four are in regions with high demand volumes, and three are from the existing Pfizer-BioNTech and Moderna open DCs. On the other hand, the other two existing DCs do not produce or ship any vaccines either because they are too close to some other open DCs in the optimal solution or they reside in the areas having low demand volumes. Moreover, the top five largest shipments are from DCs to their nearest demand sites (i.e., New York City to Philadelphia, Pleasant Prairie to Chicago, Saint Louis to Kansas City, Dallas to Denver, and San Francisco to Seattle), and each requires up to five trucks if fully occupied and the delivery time will be all within two days. All the other longer-distance shipments are in smaller volumes, requiring only one truck. For example, the longest-distance shipment is from Norwood to Seattle, which may take 4 to 5 days, but it only carries 25 doses of vaccines. Comparing Figure \ref{fig:vaccine_flow_phase1_dro} with Figure \ref{fig:vaccine_flow_phase1_sp}, DRO selects more DCs to open, among them only two are existing open ones. The number of shipment routes also increases significantly. As San Francisco is not an open DC anymore, the largest shipment occurs between Seattle and San Francisco and requires 8 trucks and 2 days to transit. The second and third largest shipments are from Philadelphia to New York City and Norwood to Boston, respectively, both requiring 3 trucks and 1 day. 

Furthermore, we present Phase 1's unsatisfied demand mean, standard deviation, and 75 to 95 percentile values in the 1000 out-of-sample scenarios according to the three approaches in Figure \ref{fig:UD_vaccine_measures}, and average unsatisfied demand percentages in each region in Figure \ref{fig:regional_US_vaccine}, respectively. From Figure \ref{fig:UD_vaccine_phase1}, DRO attains the least amount of unsatisfied demand while DT performs the worst. Comparing the results in different regions, Region 5 -- Chicago does not meet over 4\% of the demand in DT approach, and Region 8 -- Denver does not cover about 4\% of the demand in SP approach, both of which have relatively low demand volumes according to Table \ref{tab:demand_mean_vaccine}. 

\begin{figure}[ht!]
    \centering
    \begin{subfigure}[c]{.6\linewidth}
    \centering\includegraphics[width=\textwidth]{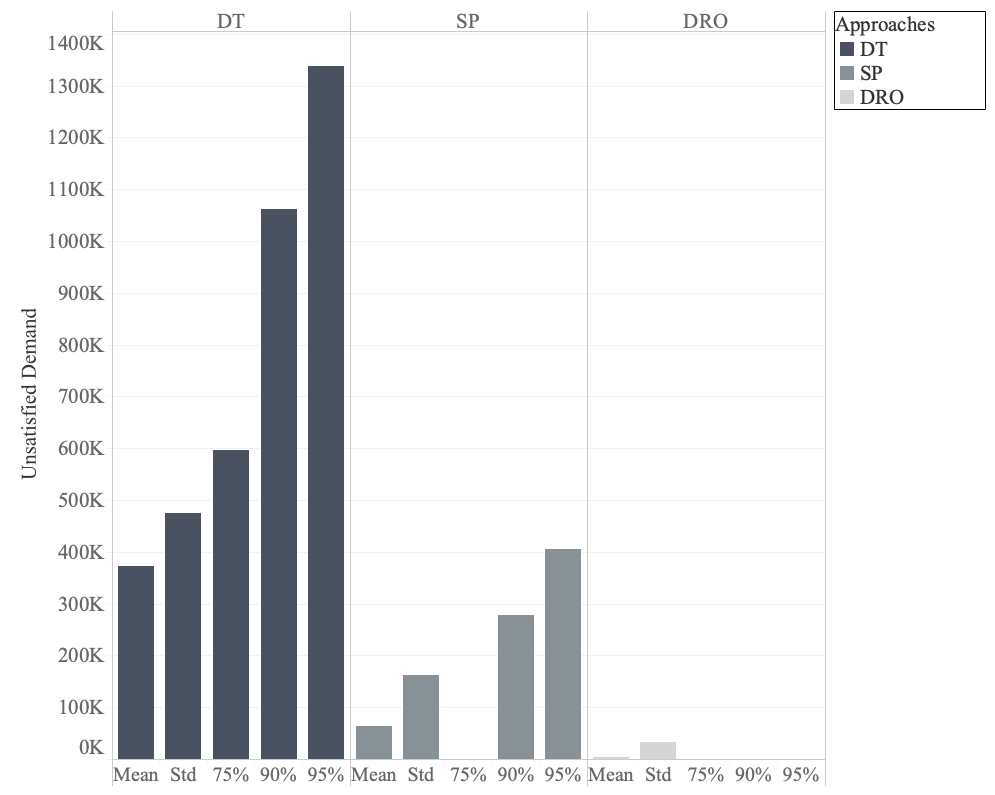}
    \caption{Unsatisfied demand statistics}\label{fig:UD_vaccine_measures}
    \end{subfigure}%
    \begin{subfigure}[c]{.4\linewidth}
    \centering\includegraphics[width=\textwidth]{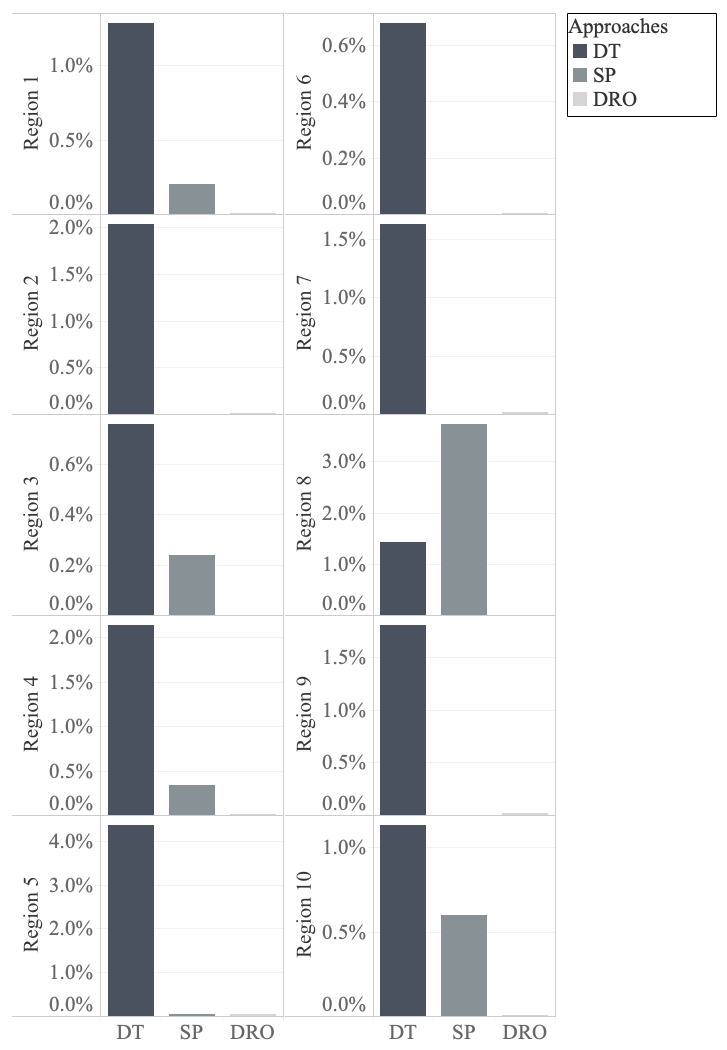}
    \caption{Regional unsatisfied demand}\label{fig:regional_US_vaccine}
    \end{subfigure}
\caption{Unsatisfied demand given by vaccine distribution in Phase 1 of different approaches.}\label{fig:UD_vaccine_phase1}
\end{figure}

In the above results, we assume that the existing 5 open DCs for Pfizer-BioNTech and Moderna have to be used but we can open additional ones if necessary. Next we examine two other cases where we either relax this constraint and allow to open facilities in any DC location options, or only open the 5 DCs. We refer to these two cases as the ``best-case'' and ``most-restrictive'' settings, respectively, and compare their Phase 2's results with the default setting in Table \ref{tab:vaccine-DC-settings}. In Column ``Overall Cost'', we display the total cost for each approach and present their percentage changes from their corresponding costs in the default setting. The last column records the total number of DCs open by each approach.

% Table generated by Excel2LaTeX from sheet 'In-sample100vaccine_biweekly'
\begin{table}[ht!]
  \centering
  \caption{Out-of-sample performance of different approaches for Phase 2 of vaccine distribution in the US under different DC settings.}
  \resizebox{\textwidth}{!}{
    \begin{tabular}{ccrrrrrrr}
    \hline
    DC Settings & Model & Operating & Capacity & Shipping & Inventory & Penalty & Overall Cost & \# DCs \\
    \hline
    \multirow{3}[0]{*}{Default} & DT    & \$194K & \$3B & \$145K & \$259 & \$231M & \$4B & 14 \\
          & SP    & \$175K & \$3B & \$126K & \$752 & \$17M & \$3B & 13 \\
          & DRO   & \$107K & \$3B & \$545K & \$519 & \$9M & \$3B & 8 \\
          \hline
    \multirow{3}[0]{*}{Best-case} & DT    & \$154K & \$3B & \$142K & \$260 & \$231M & \$4B ($-1.21\times 10^{-3}\%$) & 10 \\
          & SP    & \$154K & \$3B & \$118K & \$744 & \$17M & \$3B (+0.01\%) & 10 \\
          & DRO   & \$204K & \$3B & \$473K & \$3K & \$7M & \$3B ($-0.08\%$) & 14 \\
          \hline
    \multirow{3}[0]{*}{Most-restrictive} & DT    & \$60K & \$3B & \$1M & \$143 & \$272M & \$4B (+1.19\%) & 5 \\
          & SP    & \$60K & \$3B & \$1M & \$522 & \$18M & \$3B (+0.06\%) & 5 \\
          & DRO   & \$60K & \$3B & \$1M & \$455 & \$10M & \$3B (+0.03\%) & 5 \\
          \hline
    \end{tabular}%
    }
  \label{tab:vaccine-DC-settings}%
\end{table}%

From Table \ref{tab:vaccine-DC-settings}, after relaxing the constraints on the existing 5 open DCs for Pfizer-BioNTech and Moderna, both DT and SP reduce the operational and shipping cost in the best-case setting due to better options of DCs to open. Moreover, the number of open DCs increases from 8 to 14 in the DRO approach because it has more freedom to choose DCs. In the best-case setting, DRO opens more DCs than the other two benchmarks as the optimal solution is chosen against the worst-case scenario where the demand mean is higher than the empirical one. %As a result, it leads to much lower shipping, penalty and overall cost. 
Comparing the most-restrictive setting where only the existing 5 DCs are open with the default setting, all approaches have  increased shipping, penalty and overall costs.

Next, we vary the resource scarcity level between 10\% and 30\% of the full capacity and present the out-of-sample performance of different approaches for Phase 2 in Table \ref{tab:vaccine-scarcity}. From the table, the results with 30\% scarcity level are similar to the ones with full capacity. When the scarcity level decreases, the capacity cost drops while penalty and overall costs increase drastically. Moreover, results of the three approaches are more similar given low capacity.
% Table generated by Excel2LaTeX from sheet 'In-sample100vaccine_biweekly'
\begin{table}[ht!]
  \centering
  \caption{Out-of-sample performance of different approaches for Phase 2 of vaccine distribution in the US under different resource scarcity levels.}
    \begin{tabular}{ccrrrrrr}
    \hline
    Scarcity & Model & Operating & Capacity & Shipping & Inventory & Penalty & Overall Cost \\
    \hline
    \multirow{3}[0]{*}{30\%} & DT    & \$194K & \$3B & \$145K & \$260 & \$231M & \$4B \\
          & SP    & \$175K & \$3B & \$126K & \$752 & \$17M & \$3B \\
          & DRO   & \$131K & \$3B & \$419K & \$536 & \$10M & \$3B \\
          \hline
    \multirow{3}[0]{*}{20\%} & DT    & \$194K & \$3B & \$150K & \$15  & \$2B & \$5B \\
          & SP    & \$175K & \$3B & \$117K & \$14  & \$2B & \$5B \\
          & DRO   & \$178K & \$3B & \$159K & \$12  & \$2B & \$5B \\
          \hline
    \multirow{3}[0]{*}{10\%} & DT    & \$105K & \$2B & \$48K & \$0   & \$17B & \$19B \\
          & SP    & \$96K & \$2B & \$42K & \$0   & \$17B & \$19B \\
          & DRO   & \$105K & \$1B & \$41K & \$0   & \$17B & \$19B \\
          \hline
    \end{tabular}%
  \label{tab:vaccine-scarcity}%
\end{table}%

We also compare the vaccine allocation results given by the SP and DRO approaches with the current vaccination distribution status reported by CDC \citep{link_CDCDataTracker}. From CDC's website \citep{link_CDCPfizer, link_CDCModerna}, Pfizer-BioNTech has started to provide first-dose vaccines since December 14, 2020, while Moderna started to ship first-dose vaccines on December 21, 2020. Consequently, we let December 14, 2020 be the starting date of our planning horizon and denote December 14, 2020 -- January 11, 2021 as Phase 1, January 11, 2021 -- March 8, 2021 as Phase 2, and March 8, 2021 -- May 31, 2021 as Phase 3. We report the total number of doses that our SP and DRO approaches have distributed and CDC has delivered/administered by each checkpoint in Table \ref{tab:doses_comparison} as of September 27, 2021. We note that both of our SP and DRO models finish distributing vaccines by the end of Phase 3 (i.e., May 31, 2021) and as a result, the total number of doses distributed until September 27, 2021 remains the same with the one until May 31, 2021, while CDC still distributes vaccines during this period.

% Table generated by Excel2LaTeX from sheet 'Sheet2'
\begin{table}[ht!]
  \centering
  \caption{Total number of doses distributed in the optimal solutions of SP/DRO and CDC's delivered and administered data for each phase as of September 27, 2021}
    \begin{tabular}{lrrrr}
    \hline
    Approaches & \multicolumn{1}{l}{Until 1/11/2021} & \multicolumn{1}{l}{Until 3/8/2021} & \multicolumn{1}{l}{Until 5/31/2021} & \multicolumn{1}{l}{Until 9/27/2021} \\
    \hline
    SP Distributed   & 20M & 155M & 321M & 321M \\
    DRO Distributed  & 21M & 157M & 322M & 322M \\
    CDC Delivered   & 24M & 116M & 366M & 472M \\
    CDC Administered & 9M & 92M & 296M & 391M \\
    \hline
    \end{tabular}%
  \label{tab:doses_comparison}%
\end{table}%

From Table \ref{tab:doses_comparison}, compared with SP and DRO approaches, CDC allocates much more vaccine doses in Phase 1 and but fewer in Phase 2. By the end of Phase 3, the SP and DRO approaches have distributed about 321 millions of doses, which are much closer to the real demand (i.e., 296M) as reported by CDC's administered data. As of September 27, 2021, CDC delivered 472 millions of doses while the total number of doses received by all people in the US is only 391M, leaving about 81 millions of doses unused, becoming either disposal or inventory. Both of the SP and DRO approaches are demand-driven, as they can improve CDC vaccine allocation plans by allocating more vaccines to appropriate regions when demand is high, or by reducing allocation amounts to avoid unnecessary shipping and inventory cost when demand is low.

\subsection{COVID-19 Test Kit Allocation in Michigan, USA}
\label{sec:TestkitCaseStudy}

\subsubsection{Experimental Design and Setup}

To demonstrate that the presented generic approaches are well-suited for distributing various types of resources in different geographic scales, in this section we consider the COVID-19 test kit resource allocation problem in the State of Michigan, and divide the state into eight emergency preparedness regions according to Figure \ref{fig:MI-regions} and consider aggregated demand in these regions. (Regions 2N, 2S, 3 are renamed as Regions 2, 3 and 4 for notation simplicity in the case study.) We pick 29 test-kit suppliers in Michigan as candidate locations for setting up DCs.

\begin{figure}[ht!]
	\centering
	\includegraphics[width=0.5\textwidth]{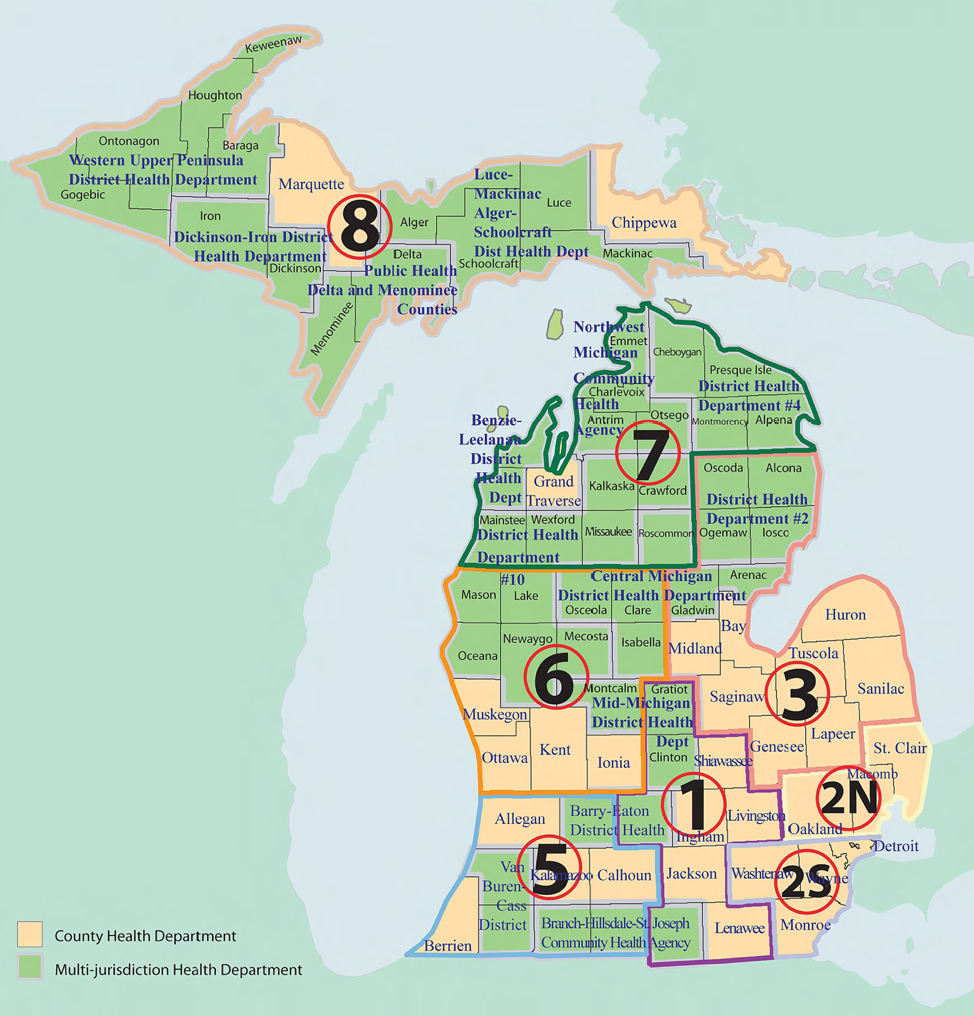}
\caption{Michigan Emergency Preparedness Regions (source: Michigan state government).}
\label{fig:MI-regions}
\end{figure}

For predicting demand and generating uncertain scenarios, % to represent the underlying uncertainty in the infections, 
we mainly follow \citet{pei2020initial}, in which the authors estimate the daily new county-level confirmed cases in the US based on a meta-population Susceptible-Exposed-Infectious-Removed (SEIR) model. The demand projections are reported in the 2.5, 25, 50, 75, 97.5 percentiles for each county everyday, and are regularly updated to capture different trends in virus transmission and intervention policies during the pandemic. To identify the demand of each emergency preparedness region, we first calculate the percentiles of the total projected confirmed cases, and divide them by our targeted positivity rate (i.e., 5\%), meaning that we aim to test 100 people if there are 5 confirmed cases. Having obtained the percentiles of the projected demand, we extract the mean and variance according to \cite{wan2014estimating} by assuming the 2.5 and 97.5 percentile values being lower and upper bounds of the support of the underlying random demand for estimation. 

We consider different planning horizons to represent various phases of virus transmission in Michigan --
%. The trend of the disease transmission in Michigan is depicted in Figure \ref{fig:MI_trend}. From Figure \ref{fig:MI_trend}, 
the number of infected cases reached the first peak in April, 2020; it was then controlled in June, 2020 with the regulations from governor's executive orders and closures of businesses and public services; around December, winter holiday gatherings increased virus spread, and thus resulted in another peak;  most recently, the re-opening of schools in September, 2021 leads to continued transmissions. Consequently, we select  April 5, 2020, June 4, 2020, December 20, 2020 and September 19, 2021 as the starting dates of the planning horizon, where we consider a two-week planning horizon with one period being one week. As a result, we examine the presented approaches over (i) the first peak (i.e., April 5, 2020 -- April 19, 2020), (ii) off-peak periods where we can still observe the effects of lockdown procedures (i.e., June 4, 2020 -- June 18, 2020), (iii) the second peak (i.e., December 20, 2020 -- January 3, 2021) and (iv) continued transmission period (i.e., September 19, 2021 -- October 3, 2021).

In terms of parameter choices, we set the upper bound $M_i$ of the capacity $h_{it}$ to the maximum production amount of supplier $i\in \mathcal{I}$ at period $t \in \mathcal{T}$ depending on the size of each supplier, and the temporal capacity of manufacturing test kits $B_t$ to the sum of maximum capacity over all DCs for all $t\in \mathcal{T}$. We also test the case where there are only scarce resources available (usually in earlier phases of the pandemic), and set the capacity $B_t$ to 10\% of the regular capacities for all $t\in \mathcal{T}$. 
The operating cost $c^o_i$ for each supplier $i\in \mathcal{I}$ is similar to the one in the vaccine distribution case. We set the unit capacity cost $c^h_i$ to \$20 for all $i\in \mathcal I$, and the unit inventory cost $c^I_{jt}$ to the average pallet storage cost divided by the number of test kits a pallet can store for all $j\in \mathcal J,\ t\in \mathcal T$. The unit shipping costs $c^s_{ijt}$ for all region $j\in \mathcal J$ and supplier $i\in \mathcal I$ are calculated as \$3 per mile times the distance in miles divided by the number of test kits a truck can carry. In the state-level test kit allocation problem, we examine three cases of the unit penalty $c_{jt}^u$ for unmet demand such that we can prioritize population groups or regions during different time periods. 
In particular, in Case (i), we set $c_{jt}^u$ to a constant, i.e., $c_{jt}^u = 100$ for all $j$ and $t$; in Case (ii), we let $c_{jt}^u$ be linearly dependent on the projected demand median for region $j$ at period $t$ and set $c_{jt}^u = d^{\text{median}}_{jt} + 10$, indicating that we add a small constant to the median as some regions' demand medians are 0 in early phases and also penalize demand loss; and in Case (iii), we  set $c_{jt}^u = 0.001 d^{\text{elder}}_{j}$ with $d^{\text{elder}}_{jt}$ being the number of people above 65 years old in region $j$. In the latter two cases, we prioritize regions with more infections and more elderly or high-risk population groups, respectively. 

To compare the severity of disease transmission in different regions and phases, we display the weekly projected demand median $d^{\text{median}}_{jt}$ for 8 regions in Michigan during the four phases as well as 65 or older populations in each region in Table \ref{tab:median}.
% Table generated by Excel2LaTeX from sheet 'Model I'
\begin{table}[ht!]
  \centering
  \caption{Projected demand median during different phases and population above 65 for 8 regions}
  \resizebox{\textwidth}{!}{
    \begin{tabular}{crrrrrrrr}
    \hline
          {Regions}& \multicolumn{1}{l}{Region 1} & \multicolumn{1}{l}{Region 2} & \multicolumn{1}{l}{Region 3} & \multicolumn{1}{l}{Region 4} & \multicolumn{1}{l}{Region 5} & \multicolumn{1}{l}{Region 6} & \multicolumn{1}{l}{Region 7} & \multicolumn{1}{l}{Region 8} \\
          \hline
    {Demand in Apr. 2020} & 153   & 3.6K  & 2.2K  & 107   & 2   & 407   & 31    & 0 \\
    {Demand in June 2020} & 43   & 236  & 516  & 61   & 112   & 443   & 0    & 0 \\
    {Demand in Dec. 2020} & 55.5K & 100.0K & 93.3K & 72.7K & 52.1K & 76.5K & 18.8K & 12.7K\\
    {Demand in Sept. 2021} & 48.5K & 76.0K & 72.2K & 51.0K & 56.4K & 83.3K & 14.3K & 15.8K\\
    \hline
    \hline
    Population above 65 & 181.8K & 360.2K & 349.2K & 216.9K & 143.5K & 237.2K & 84.2K & 66.5K\\
    \hline
    \end{tabular}%
    }
  \label{tab:median}%
\end{table}%
As can be seen from Table \ref{tab:median}, Regions 7 and 8 have significantly fewer number of new cases each day and senior populations than other regions, while Regions 2 and 3 have the highest projected demand and senior population and thus will be most prioritized when we follow the penalty cost cases (ii) and (iii). 

\subsubsection{Results}

We first obtain optimal solutions of the three models (DT, SP, DRO) based on 100 in-sample scenarios and 1000 out-of-sample scenarios following the same normal distribution with the estimated mean and variance. For DRO, we set parameters in the ambiguity set as $\epsilon^{\mu}_{jt}=0.5\mu_{jt},\ \underline{\epsilon}^{S}_{jt}=0.1,\ \bar{\epsilon}^{S}_{jt}=2$ for all $j\in \mathcal{J},\ t\in\mathcal{T}$ in June 2020, December 2020 and September 2021 and set $\epsilon^{\mu}_{jt}=\mu_{jt},\ \underline{\epsilon}^{S}_{jt}=0.01,\ \bar{\epsilon}^{S}_{jt}=10$ for all $j\in \mathcal{J},\ t\in\mathcal{T}$ in April 2020. % due to infeasibility issues of the inner maximization problem in Model \eqref{eq:Type-1-DR-Obj}.
Letting the penalty parameter $c^u_{jt} = 0.001d^{\text{elder}}_{j}$ for each region $j \in \mathcal{J}$ (i.e., Case (iii) penalty setting), we summarize out-of-sample performance given by solutions of the three approaches under different resource settings in Table \ref{tab:testing_summary}, where we mark the percentage increase compared with the optimal one in parentheses. The last column in Table \ref{tab:testing_summary} records the in-sample computational time in seconds. The raw cost breakdowns of each approach are presented in Tables \ref{tab:testkitCosts-ample} and \ref{tab:testkitCosts-scarce} in the Appendix. %\ref{sec:Appendix-costs}.

% Table generated by Excel2LaTeX from sheet 'Summary'
\begin{table}[ht!]
  \centering
  \caption{Out-of-sample performance of different approaches in the COVID-19 test kit distribution example}
  \resizebox{\textwidth}{!}{
    \begin{tabular}{cc|rr|rr|r}
    \hline
    \multirow{2}[0]{*}{Phases} & \multirow{2}[0]{*}{Approach} & \multicolumn{2}{c|}{Unsatisfied Demand} & \multicolumn{2}{c|}{Overall Cost} & \multirow{2}[0]{*}{Time (sec.)}  \\
          &       & \multicolumn{1}{r}{{Full Resources}} & \multicolumn{1}{r|}{{Scarce Resources}} & \multicolumn{1}{r}{{Full Resources}} & \multicolumn{1}{r|}{{Scarce Resources}} & \\
          \hline
    \multirow{3}[0]{*}{Apr. 2020} & DT & {76K ($+5.86\times 10^5\%$)}     & {76K ($+3.89\times 10^3\%$)}    & {\$29M (+68.27\%)}     & {\$29M (+68.27\%)} & 0.1 \\
          & SP    & {3K ($+2.00\times 10^4\%$)}     & {3K ($+36\%$)}     & {\$17M}     & {\$17M} & 9.5 \\
          & DRO   & {13}     & {2K}     & {\$20M (+17.02\%)}     & {\$17M (+1.04\%)} & 12.1 \\
          \hline
    \multirow{3}[0]{*}{Jun. 2020} & DT & {5K ($+4.24\times 10^4\%$)}     &  {5K ($+4.24\times 10^4\%$)}    & {\$2M (+28.79\%)}    & {\$2M (+28.79\%)} & 0.1\\
          & SP    & {691 ($+5.78\times 10^3\%$)}     & {691 ($+5.78\times 10^3\%$)}    & {\$1M}     & {\$1M} & 9.4\\
          & DRO   & {12}     & {12}      & {\$2M (+18.69\%)}      & {\$2M (+18.69\%)} & 20.6\\
          \hline
    \multirow{3}[0]{*}{Dec. 2020} & DT & {40K ($+2.41\times 10^4\%$)}     & {150K}     & {\$22M (+9.75\%)}     & {\$34M} & 0.1\\
          & SP    & {6K ($+3.89\times 10^3\%$)}     & {150K}     & {\$20M}      & {\$34M ($+2.80\times 10^{-3}\%$)}  & 5.2\\
          & DRO   & {164}      & {150K}      & {\$22M (+10.49\%)}      & {\$34M ($+5.03\times 10^{-3}\%$)} & 12.0 \\
          \hline
        \multirow{3}[0]{*}{{Sept. 2021}} & DT & {39K}     & {55K}     & {\$21M (+9.44\%)}     & {\$22M ($+1.32\times 10^{-3}\%$)} & 0.1 \\
          & SP    & {10K}     & {55K}     & {\$19M}      & {\$22M } & 9.8 \\
          & DRO   & {0}      & {55K}      & {\$21M (+12.17\%)}      & {\$22M ($+9.17\times 10^{-3}\%$)} & 13.0 \\
          \hline      
    \end{tabular}%
    }
  \label{tab:testing_summary}%
\end{table}%

From Table \ref{tab:testing_summary}, in the full resource level, DRO always attains the least amount of unsatisfied demand, while DT has the most. Under scarce resources (i.e., 10\% of the full resource level), the results almost remain the same in April and June 2020, because most of the approaches do not use up 10\% of the total resources. However, when the resource capacity becomes tight such as in December 2020 and September 2021, the results of the three approaches are almost identical as there is not much flexibility in producing and shipping test kits. In terms of overall cost, SP outperforms the other two models in most of the settings. 

We display the detailed unsatisfied demand percentages in each region under different cases of penalty cost patterns in Figure \ref{fig:regional UD_testing}. Comparing Figure \ref{fig:regional UD_testing} with Table \ref{tab:median}, for Cases (ii) and (iii) penalty settings, the demand in Regions 2, 3, 4 and 6 is all satisfied, as these regions have more projected demand in December and more senior populations. On the other hand, Regions 7 and 8 always have the highest unsatisfied demand as they are farther away from the DCs and have less projected demand.

\begin{figure}[ht!]
    \centering
    \includegraphics[width=0.6\textwidth]{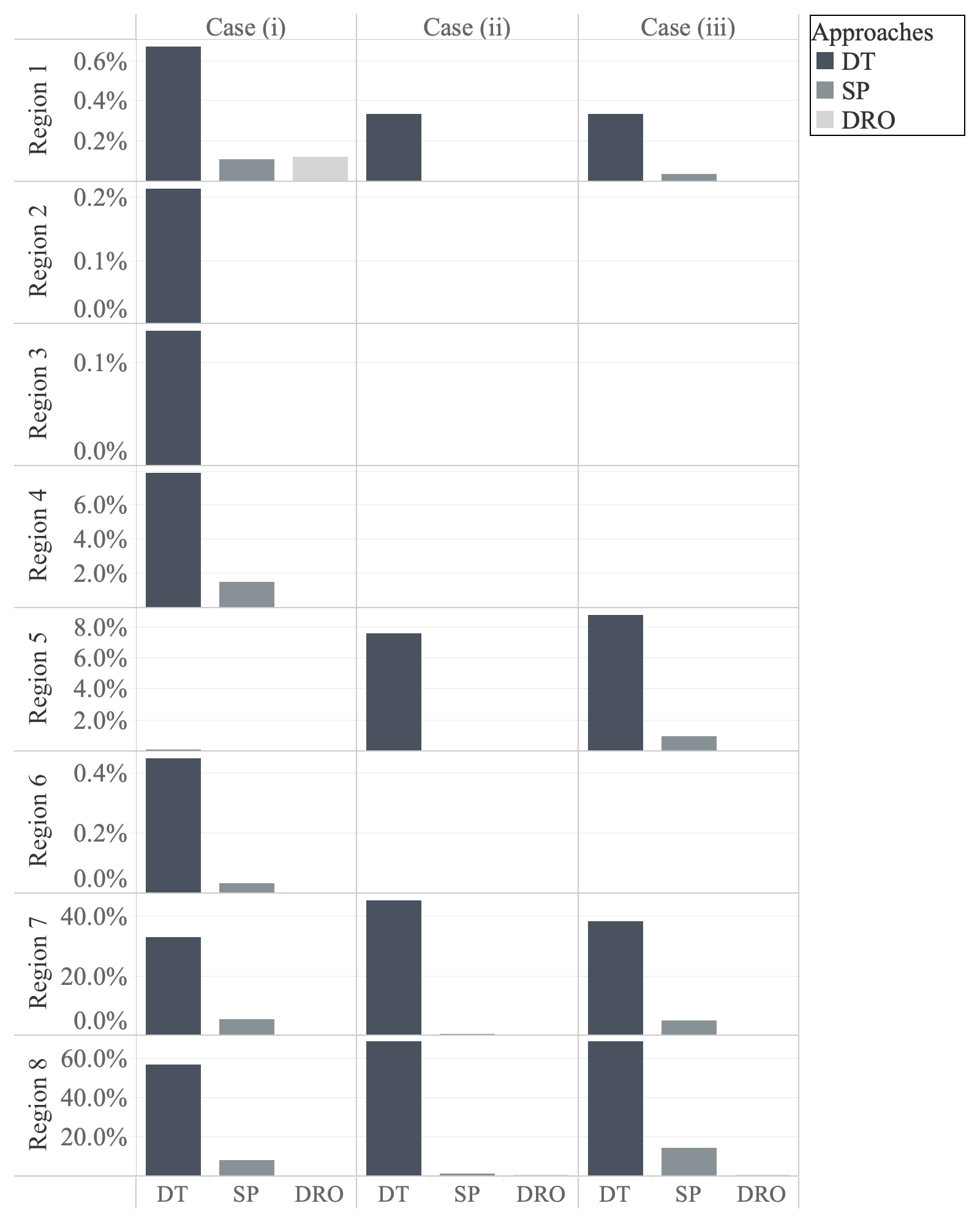}
    \caption{Unsatisfied demand percentages in each region in December with full resources given Cases (i), (ii), (iii) unit penalty.}
    \label{fig:regional UD_testing}
\end{figure}

Next, we present the cost breakdown over different phases in Figure \ref{fig:cost_testing}. From Figure \ref{fig:cost_testing}, under full resources, DRO can satisfy almost all the demand, while DT still leaves thousands of people untested. However, when we limit the resources to be 10\% of the regular amounts (see gray dashed line in Figure \ref{fig:cost_testing}), all capacity costs that exceed the gray line will be truncated to the one that uses only 10\% of the original $B_t$. 

\begin{figure}[ht!]
    \centering
    \includegraphics[width=\textwidth]{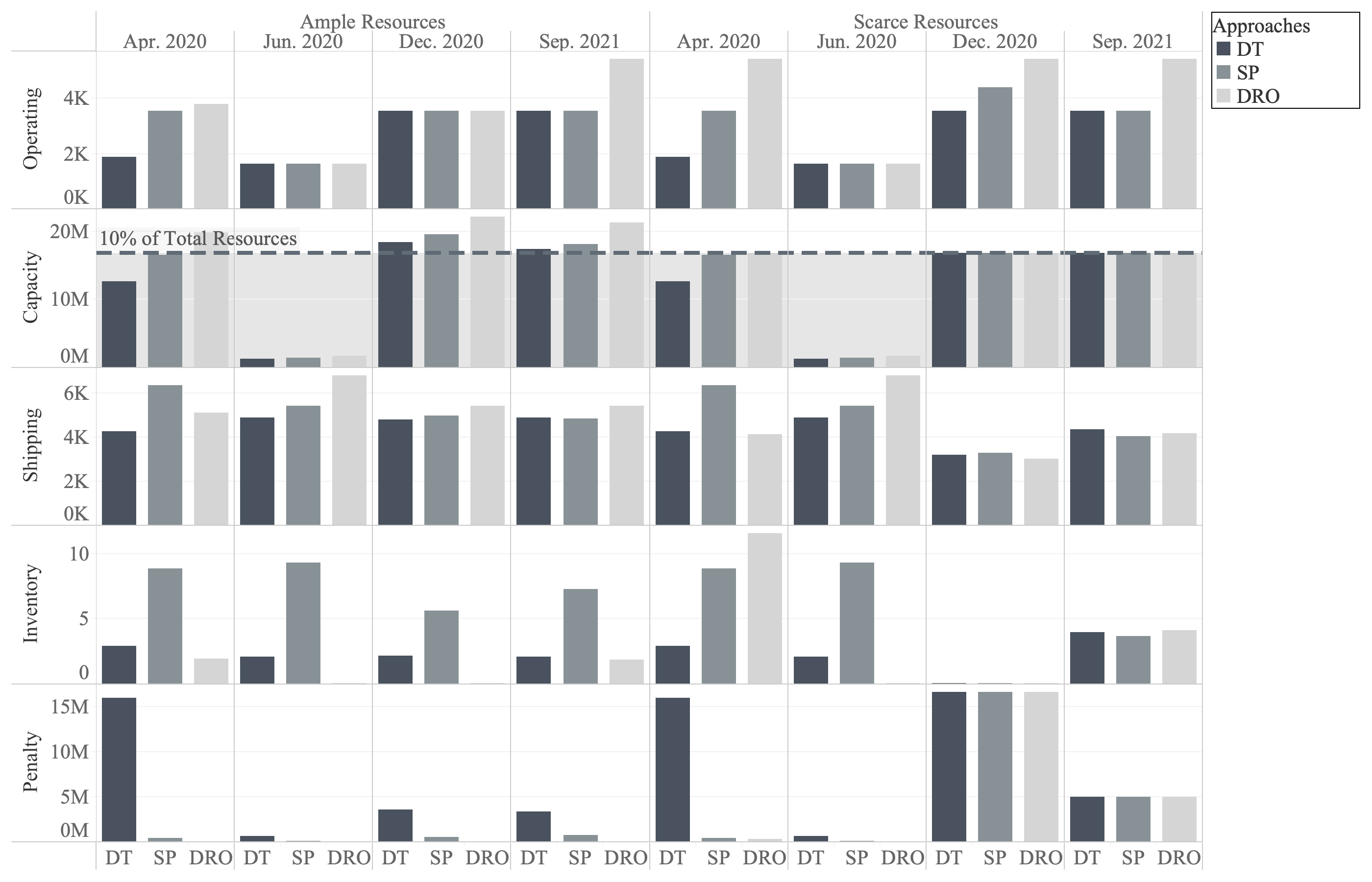}
    \caption{ Cost breakdown of test kit distribution over different phases under ample and scarce resources (in dollars).}
    \label{fig:cost_testing}
\end{figure}

\section{Conclusion}
\label{sec:concl}

In this paper, we presented a generic framework to model resource distribution for epidemic response under spatiotemporal demand uncertainties for these resources. Depending on the level of statistical information available for characterizing the uncertainties, which are impacted by various factors including infection trends and demographic behavior, we proposed a stochastic programming and a distributionally robust optimization approach to optimize the locations and capacities of DCs, together with shipping, inventory and demand loss variables. As the proposed optimization frameworks can be applied to different scales of resource distribution, we presented two case studies using instances of COVID-19 vaccine distribution in the US and the analogous test kit distribution in the State of Michigan. Furthermore, we considered  different phases of the pandemic, ample or scarce resources, and three cases of unmet demand penalty settings to compare the results of deterministic, stochastic and robust optimization approaches. 

Our approaches aim efficient and fair distribution of resources by allowing prioritization of regions with more vulnerable population groups or with higher infection susceptibility. The case studies demonstrated the importance of incorporating demand uncertainty in these planning problems as the stochastic programming and distributionally robust optimization approaches outperform the deterministic one in terms of cost and demand coverage. The distributionally robust approach provided a better most-restrictive performance in terms of unvaccinated or untested people who qualify, with an overall cost higher than the one of the stochastic programming approach.Our comparison of the proposed approaches with the CDC's current distribution status of COVID-19 vaccines shows that our approaches are demand-driven, satisfy more demand in the earlier phases and also prevent unnecessary production, shipping and inventory costs in the later phases of the pandemic when demand drops. Furthermore, the case study of test kit distribution demonstrated the prioritization of resource allocation depending on the infection trend and vulnerable population percentages over the studied regions. 

 As a limitation of the work, we did not incorporate the time difference of demand for dose 1 and dose 2 of the Pfizer-BioNTech or Moderna vaccines into the optimization models, but only consider the aggregated total demand. For future research, it will be interesting to consider different uncertainties for dose-1 and dose-2 demand, and also incorporate the uncertain time between taking two doses into the SP and DRO models. One can also build divergence-based ambiguity sets rather than using the moment information or consider continuous support case in the DRO approach.  

\section*{Acknowledgement} 
The authors gratefully acknowledge the partial support from U.S.\ National Science Foundation (NSF) grants \#CMMI-1727618, 2041745 and Department of Energy (DoE) grant \#DE-SC0018018. 

%\bibliographystyle{apalike}
%\bibliography{sshen,beste,Xian,covid19}

\section*{Appendix: Parameter Sources and More Results}
\label{sec:Appendix}
We first present the list of parameters and decision variables of the SMIP model \eqref{eq:TS-MILP}, and sources of parameters in vaccine distribution and test kit allocation problems in Tables \ref{tab:SetsDecVarsParameters} and \ref{tab:parameterSources}, respectively.

\begin{table}[ht!]
    \centering
    \caption{List of sets, parameters and decision variables.}
    \resizebox{\textwidth}{!}{
    \begin{tabular}{r|l}
    \hline
       Name  & Definition \\
       \hline
       \bf{Sets:} & \\
       $\mathcal{I}$ & Sets of potential sites for DCs\\
       $\mathcal{J}$ & Sets of demand sites \\
       $\mathcal{T}$ & Sets of time periods in the planning horizon \\
       \hline
       \bf{Parameters:} & \\
       $c^o_i$  &  Cost of operating DC at location $i$\\
       $c^h_{i}$ & Cost of installing unit capacity for DC at location $i$\\
       $c^s_{ijt}$ & Cost of shipping an unit from DC at location $i$ to demand site $j$ at period $t$ \\
       $c^I_{jt}$ & Penalty cost of unit unsatisfied demand at demand site $j$ at period $t$ \\
       $c^u_{jt}$ & Holding cost of an unit at demand site $j$ at period $t$\\
       $B_t$ & Total resource capacity at period $t$ \\
       $M_i$ & Capacity limit of DC at location $i$ \\
       $d_{jt}$ & Demand of demand site $j$ at period $t$ \\
       \hline
       \bf{Decision variables:} & \\
       $x_i$ & Binary variable denoting whether DC at location $i$ is open \\
       $h_{it}$ & Capacity of DC at location $i$ at period $t$\\
       $s_{ijt}$ & Amount of resources sent from DC at location $i$ to demand site $j$ at period $t$ \\
       $I_{jt}$ & Amount of inventory at demand site $j$ at period $t$ \\
       $u_{jt}$ & Unsatisfied demand amount at demand site $j$ at period $t$ \\
       \hline
    \end{tabular}
    }
    \label{tab:SetsDecVarsParameters}
\end{table}

\begin{table}[ht!]
    \centering
    \caption{Parameter sources in vaccine distribution and test kit allocation}
    \resizebox{\textwidth}{!}{
    \begin{tabular}{r|ll}
    \hline
       Parameter  & Vaccine Distribution & Test Kit Allocation \\
       \hline
       $c^o_i$  & {\href{https://www.loopnet.com/michigan/kentwood_warehouses-for-lease/?sort=1:1}{\underline{Warehouse rental price}}}&  {\href{https://www.loopnet.com/michigan/kentwood_warehouses-for-lease/?sort=1:1}{\underline{Warehouse rental price}}}\\
       $c^h_{i}$ & \href{https://www.biospace.com/article/comparing-covid-19-vaccines-pfizer-biontech-moderna-astrazeneca-oxford-j-and-j-russia-s-sputnik-v/}{\underline{COVID-19 vaccine cost}} &\href{https://coronavirus.jhu.edu/from-our-experts/q-and-a-how-much-does-it-cost-to-get-a-covid-19-test-it-depends}{\underline{COVID-19 test kit cost}}\\
       $c^s_{ijt}$ & \href{https://www.freightwaves.com/news/understanding-total-operating-cost-per-mile}{\underline{Vans travel cost}} + \href{https://www.modernatx.com/covid19vaccine-eua/providers/storage-handling}{\underline{Vaccine dimensions}} +
       \href{https://path.azureedge.net/media/documents/TS_opt_in_country_transport_rpt.pdf}{\underline{Vaccine deliver cost}} &\href{https://www.freightwaves.com/news/understanding-total-operating-cost-per-mile}{\underline{Vans travel cost}} +
       \href{https://www.vanwisegroup.com/news/top-5-vans-with-the-largest-capacity/}{\underline{Vans capacity}}\\
       $c^I_{jt}$ & \href{https://www.warehousingcompanies.net/resources/prices-for-storing-product-in-a-warehouse/}{\underline{Warehouse storage cost}} + \href{https://americanbiotechsupply.com/refrigerators/products/pharmacy/premier/36-cu.-ft.-pharmacy-glass-door-refrigerator}{\underline{Energy cost}} &\href{https://www.warehousingcompanies.net/resources/prices-for-storing-product-in-a-warehouse/}{\underline{Warehouse storage cost}} +
       \href{https://www.copanusa.com/covid-19-sample-collection-kits-for-upper-respiratory-tract-specimens/}{\underline{COVID-19 test kit dimension}}\\
       $c^u_{jt}$ & 100 & Three penalty settings as discussed in the paper\\
       \hline
    \end{tabular}
    }
    \label{tab:parameterSources}
\end{table}

Next, we present the detailed results of cost breakdown in the national vaccine allocation and state-level test kit distribution problems under different phases, resource settings and approaches.

\begin{table}[ht!]
  \centering
  \caption{Cost breakdown of vaccine distribution under full resources, different phases and approaches}
    \begin{tabular}{lrrrrrrr}
    \hline
    Phases & Approaches & \multicolumn{1}{r}{Operating} & \multicolumn{1}{r}{Capacity} & \multicolumn{1}{r}{Shipping} & \multicolumn{1}{r}{Inventory} & \multicolumn{1}{r}{Penalty} & \multicolumn{1}{r}{Overall Cost} \\
    \hline
    \multirow{3}[0]{*}{Phase 1} & DT    & \multicolumn{1}{r}{\$97K} & \multicolumn{1}{r}{\$486M} & \multicolumn{1}{r}{\$39K} & \multicolumn{1}{r}{\$34} & \multicolumn{1}{r}{\$37M} & \multicolumn{1}{r}{\$524M} \\
          & SP    & \multicolumn{1}{r}{\$63K} & \multicolumn{1}{r}{\$496M} & \multicolumn{1}{r}{\$48K} & \multicolumn{1}{r}{\$155} & \multicolumn{1}{r}{\$6M} & \multicolumn{1}{r}{\$503M} \\
          & DRO   & \multicolumn{1}{r}{\$76K} & \multicolumn{1}{r}{\$515M} & \multicolumn{1}{r}{\$55K} & \multicolumn{1}{r}{\$80} & \multicolumn{1}{r}{\$292K} & \multicolumn{1}{r}{\$516M} \\
          \hline
    \multirow{3}[0]{*}{Phase 2} & DT    & \multicolumn{1}{r}{\$194K} & \multicolumn{1}{r}{\$3B} & \multicolumn{1}{r}{\$145K} & \multicolumn{1}{r}{\$259} & \multicolumn{1}{r}{\$232M} & \multicolumn{1}{r}{\$4B} \\
          & SP    & \multicolumn{1}{r}{\$175K} & \multicolumn{1}{r}{\$3B} & \multicolumn{1}{r}{\$125K} & \multicolumn{1}{r}{\$752} & \multicolumn{1}{r}{\$17M} & \multicolumn{1}{r}{\$3B} \\
          & DRO   & {\$107K} & {\$3B} & {\$545K} & {\$519} & {\$9M} & {\$3B} \\
          \hline
    \multirow{3}[0]{*}{Phase 3} & DT    & \multicolumn{1}{r}{\$290K} & \multicolumn{1}{r}{\$4B} & \multicolumn{1}{r}{\$180K} & \multicolumn{1}{r}{\$569} & \multicolumn{1}{r}{\$262M} & \multicolumn{1}{r}{\$4B} \\
          & SP    & \multicolumn{1}{r}{\$263K} & \multicolumn{1}{r}{\$4B} & \multicolumn{1}{r}{\$152K} & \multicolumn{1}{r}{\$2K} & \multicolumn{1}{r}{\$31M} & \multicolumn{1}{r}{\$4B} \\
          & DRO   & {\$259K} & {\$4B} & {\$411K} &{\$935} & {\$45M} & {\$4B} \\
          \hline
    \end{tabular}%
  \label{tab:vaccineCosts-ample}%
\end{table}%

\begin{table}[ht!]
  \centering
  \caption{Cost breakdown of vaccine distribution under scarce resources, different phases and approaches}
    \begin{tabular}{lrrrrrrr}
    \hline
    Phases & Approaches & \multicolumn{1}{r}{Operating} & \multicolumn{1}{r}{Capacity} & \multicolumn{1}{r}{Shipping} & \multicolumn{1}{r}{Inventory} & \multicolumn{1}{r}{Penalty} & \multicolumn{1}{r}{Overall Cost} \\
    \hline
    \multirow{3}[0]{*}{Phase 1} & DT    & \multicolumn{1}{r}{\$97K} & \multicolumn{1}{r}{\$486M} & \multicolumn{1}{r}{\$39K} & \multicolumn{1}{r}{\$34} & \multicolumn{1}{r}{\$37M} & \multicolumn{1}{r}{\$524M} \\
          & SP    & \multicolumn{1}{r}{\$63K} & \multicolumn{1}{r}{\$496M} & \multicolumn{1}{r}{\$52K} & \multicolumn{1}{r}{\$99} & \multicolumn{1}{r}{\$7M} & \multicolumn{1}{r}{\$503M} \\
          & DRO   & \multicolumn{1}{r}{\$54K} & \multicolumn{1}{r}{\$512M} & \multicolumn{1}{r}{\$58K} & \multicolumn{1}{r}{\$64} & \multicolumn{1}{r}{\$687K} & \multicolumn{1}{r}{\$513M} \\
          \hline
    \multirow{3}[0]{*}{Phase 2} & DT    & \multicolumn{1}{r}{\$105K} & \multicolumn{1}{r}{\$2B} & \multicolumn{1}{r}{\$48K} & \multicolumn{1}{r}{\$0} & \multicolumn{1}{r}{\$17B} & \multicolumn{1}{r}{\$19B} \\
          & SP    & \multicolumn{1}{r}{\$96K} & \multicolumn{1}{r}{\$2B} & \multicolumn{1}{r}{\$42K} & \multicolumn{1}{r}{\$0} & \multicolumn{1}{r}{\$17B} & \multicolumn{1}{r}{\$19B} \\
          & DRO   & {\$105K} & {\$2B} & {\$41K} & {\$0} & {\$17B} & {\$19B} \\
          \hline
    \multirow{3}[0]{*}{Phase 3} & DT    & \multicolumn{1}{r}{\$256K} & \multicolumn{1}{r}{\$3B} & \multicolumn{1}{r}{\$122K} & \multicolumn{1}{r}{\$0} & \multicolumn{1}{r}{\$13B} & \multicolumn{1}{r}{\$16B} \\
          & SP    & \multicolumn{1}{r}{\$175K} & \multicolumn{1}{r}{\$3B} & \multicolumn{1}{r}{\$134K} & \multicolumn{1}{r}{\$0} & \multicolumn{1}{r}{\$13B} & \multicolumn{1}{r}{\$16B} \\
          & DRO   & {\$209K} & {\$3B} & {\$221K} & {\$0} & {\$13B} & {\$16B} \\
          \hline
    \end{tabular}%
  \label{tab:vaccineCosts-scarce}%
\end{table}%

\begin{table}[ht!]
  \centering
  \caption{Cost breakdown of test kit distribution under full resources, different phases and approaches}
    \begin{tabular}{lrrrrrrr}
    \hline
    Phases & \multicolumn{1}{c}{Approaches} & Operating & Capacity & Shipping & Inventory & Penalty & Overall Cost \\
    \hline
    \multirow{3}[0]{*}{Apr. 2020} & DT    &  \$2K  &  \$13M  &  \$4K  &  \$3  &  \$16M  & \$29M \\
          & SP    & \$4K  & \$17M & \$6K  & \$9  & \$395K  & \$17M \\
          & DRO   & \$4K  & \$20M & \$5K  & \$2   & \$2K  & \$20M \\
          \hline
    \multirow{3}[0]{*}{Jun. 2020} & DT    & \$2K  & \$1M  & \$5K  & \$2   & \$638K & \$2M \\
          & SP    & \$2K  & \$1M  & \$5K  & \$9  & \$76K & \$1M \\
          & DRO   & \$2K  & \$2M  & \$7K  & \$0   & \$1K  & \$2M \\
          \hline
    \multirow{3}[0]{*}{Dec. 2020} & DT    & \$4K  & \$18M & \$5K  & \$2   & \$4M & \$22M \\
          & SP    & \$4K  & \$20M & \$5K  & \$6   & \$535K & \$20M \\
          & DRO   & \$4K  & \$22M & \$5K  & \$0   & \$12K & \$22M \\
          \hline
    \multirow{3}[0]{*}{{Sept. 2021}} & DT    & \$4K  & \$17M & \$5K  & \$2   & \$3M & \$21M \\
          & SP    & \$4K  & \$18M & \$5K  & \$7   & \$811K & \$19M \\
          & DRO   & \$5K  & \$21M & \$5K  & \$2   & \$0 & \$21M \\
          \hline
    \end{tabular}%
  \label{tab:testkitCosts-ample}%
\end{table}%

\begin{table}[ht!]
  \centering
  \caption{Cost breakdown of test kit distribution under scarce resources, different phases and approaches}
    \begin{tabular}{lrrrrrrr}
    \hline
    Phases & \multicolumn{1}{c}{Approaches} & Operating & Capacity & Shipping & Inventory & Penalty & Overall Cost \\
    \hline
    \multirow{3}[0]{*}{Apr. 2020} & DT    &  \$2K  &  \$13M  &  \$4K  &  \$3  &  \$16M  & \$29M \\
          & SP    & \$4K  & \$17M & \$6K  & \$9  & \$395K  & \$17M \\
          & DRO   & \$5k  & \$17M & \$4K  & \$12   & \$282K & \$17M \\
          \hline
    \multirow{3}[0]{*}{Jun. 2020} & DT    & \$2K  & \$1M  & \$5K  & \$2   & \$638K & \$2M \\
          & SP    & \$2K  & \$1M  & \$5K  & \$9  & \$76K & \$1M \\
          & DRO   & \$2K  & \$2M  & \$7K  & \$0   & \$1K  & \$2M \\
          \hline
    \multirow{3}[0]{*}{Dec. 2020} & DT    & \$4K  & \$17M & \$3K  & \$0   & \$17M & \$34M \\
          & SP    & \$4K  & \$17M & \$3K  & \$0   & \$17M & \$34M \\
          & DRO   & \$5K  & \$17M & \$3K  & \$0   & \$17M & \$34M \\
          \hline
    \multirow{3}[0]{*}{{Sept. 2021}} & DT    & \$4K  & \$17M & \$4K  & \$4   & \$5M & \$22M \\
          & SP    & \$4K  & \$17M & \$4K  & \$4   & \$5M & \$22M \\
          & DRO   & \$5K  & \$17M & \$4K  & \$4   & \$5M & \$22M \\
          \hline
    \end{tabular}%
  \label{tab:testkitCosts-scarce}%
\end{table}%

\end{document}